\documentclass[reqno]{amsart}
\usepackage[foot]{amsaddr}

\usepackage[hidelinks]{hyperref}
\usepackage{url}
\usepackage{amssymb}
\usepackage{amscd}
\usepackage{amsthm}
\usepackage{setspace}
\usepackage{enumerate}
\usepackage{graphicx}
\usepackage{mathtools}
\usepackage{tcolorbox} 
\usepackage{siunitx}
\usepackage{tikz-cd}
\usepackage{bm,blkarray}
\numberwithin{equation}{section}
\usepackage[tableposition=top]{caption}
\usepackage{booktabs,dcolumn}
\usepackage{subcaption} 
\renewcommand\P{\mathbb{P}}
\newcommand\E{\mathbb{E}}

\newcommand\R{\mathbb{R}}

\newcommand{\GL}{\mathop{\rm GL}}

\newcommand{\growth}{\mathrm{growth}}

\title[Growth Factor of Random Matrices]{On the Growth Factor of Random Matrices}
\author{John Urschel}
\address{Department of Mathematics, Massachusetts Institute of Technology, \newline \indent Cambridge, MA, 02139 USA.}
\email{urschel@mit.edu}

\newtheorem{theorem}{Theorem}[section]

\newtheorem{lemma}[theorem]{Lemma}
\newtheorem{conjecture}[theorem]{Conjecture}
\newtheorem{proposition}[theorem]{Proposition}

\newtheorem{corollary}[theorem]{Corollary}

\begin{document}

\begin{abstract}
Gaussian elimination is the oldest and most popular method for solving a linear system. Its numerical stability for a given matrix is controlled by the growth factor, a measure of how large entries can become during elimination. In this work, we prove a number of new results regarding the growth factor of random matrices. First, we provide tight estimates for the growth factor of a matrix preconditioned by a Haar orthogonal matrix without pivoting and characterize the asymptotic distribution of the growth factor of Gaussian matrices without pivoting. Second, and most notably, we prove that the growth factor of an $n \times n$ Gaussian matrix under partial pivoting is rarely much larger than $\sqrt{n}$, resolving an old conjecture of Nick Trefethen. The same techniques used to prove this conjecture also provide an improved smoothed analysis of the growth factor under partial pivoting.
\end{abstract}

\maketitle

\section{Introduction}
Gaussian elimination is the premier technique for solving a linear system $A \bm x = \bm b$, by factoring $A$ into the product of a lower and upper triangular matrix. In particular, given an $n \times n$ matrix $A$ with non-zero leading principal minors, the procedure performs a series of rank-one transformations
\[A \rightarrow A/A_{[1],[1]} \rightarrow A/A_{[2],[2]} \rightarrow \ldots \rightarrow A/A_{[n-1],[n-1]}.\]
The resulting LU factorization is encoded in the first row and column of the iterates $A/A_{[k],[k]}$,
\begin{equation}\label{eq:landu}L \bm e_{k+1} = \begin{bmatrix} 0 \\ (A / A_{[k],[k]})_{11}^{-1}
(A / A_{[k],[k]}) \bm e_1\end{bmatrix} \qquad \text{and} \qquad \bm e_{k+1}^T U = \begin{bmatrix} 0 & \bm e_1^T (A / A_{[k],[k]}) \end{bmatrix},\end{equation}
where $M / X$ is the Schur complement of $M$ with respect to the submatrix $X$, and $[n]=\{1,\ldots,n\}$. While this algorithm is exact in exact arithmetic, issues may still arise in finite precision due to the accumulation of round-off error. The error in Gaussian elimination in finite precision is controlled by the growth factor
\[\growth(A) = \max \left\{ \|L\|_{\max},\frac{\|U\|_{\max}}{\|A\|_{\max}}\right\},\]
where $\|\cdot\|_{\max}$ is the maximum magnitude entry. If $\growth(A)$ is not too large, then the LU factorization is stable\footnote{One subtle point, which is often glossed over in the analysis of Gaussian elimination, is that the growth factor needed for error analysis is the growth factor under floating point arithmetic. In some cases, this quantity can be quite different from the growth factor in exact arithmetic, due to an accumulation of round-off error or a change in pivot selection. Still, for mathematical neatness, the majority of results regarding the growth factor are for exact arithmetic. There is little harm in this: The worst-case growth factor for partial pivoting is the same in exact and floating point arithmetic, and, given a reasonable number of bits of precision, the worst-case growth factor for complete pivoting is barely changed by floating point arithmetic \cite[Theorem 1.5]{edelman2024some}, and the growth factor of a Gaussian matrix under partial pivoting is increased by at most a small multiplicative constant with high probability \cite[Section 7]{huang2024average}.} (e.g., roughly $\log (n \growth (A))$ bits of precision suffice); see \cite[Theorem 9.3 \& Lemma 9.6]{higham2002accuracy} for details. For this reason, a great deal of research regards the behavior of the growth factor under various pivoting and preconditioning strategies. 

\subsection{Partial Pivoting} Gaussian elimination with partial pivoting is the standard linear system solver in nearly every programming language, including the $\mathrm{numpy.linalg.solve}$ in Python and the backslash command for a general matrix in Matlab and Julia. The partial pivoting procedure is a greedy rule for reordering rows during elimination to avoid creating large entries. In particular, at each step of elimination, the first row is (possibly) swapped with another, so that the largest magnitude entry in the first column lies in the first row. While algorithmically this is done iteratively, it is represented mathematically by left-multiplication by a single permutation matrix $P$ that satisfies 
\begin{equation}\label{eq:partial}
|((PA)/(PA)_{[k],[k]})_{11}| = \|((PA)/(PA)_{[k],[k]}) \bm e_1\|_{\infty} \qquad \text{for} \quad k = 0,1,\ldots, n-1.\end{equation}
Let $\Pi(\cdot): \GL_n(\R) \rightarrow \GL_n(\R)$ denote the partial pivoting map, which left-multiplies a non-singular matrix $A$ by the permutation matrix satisfying \eqref{eq:partial}. If there are multiple, $\Pi(\cdot)$ breaks ties at each stage of elimination by choosing the eligible row with smallest index in the current row ordering. However, in this paper, such ambiguities occur with probability zero.

The growth factor under partial pivoting can be exponentially large. A fairly simple doubling argument shows that the growth factor under partial pivoting is at most $2^{n-1}$. In 1961, Wilkinson produced a matrix construction that achieves this quantity in every dimension \cite{wilkinson1961error}. In spite of this, the relatively low computational cost and memory access requirements for implementing the partial pivoting rule \eqref{eq:partial}, and the apparent rarity of a large growth factor, have made Gaussian elimination with partial pivoting the preferred method for solving a linear system. In an attempt to provide an explanation for this apparent rarity, Trefethen and Schreiber performed an empirical average-case analysis of Gaussian elimination under various pivoting strategies \cite{trefethen1990average}. A matrix $G \in \R^{m \times n}$ is said to be Gaussian if its entries are independent standard normal random variables. Inspired by their analysis, Edelman performed further numerical experiments for Gaussian matrices under partial pivoting. He noted that the growth factor appeared to grow like $n^{\frac12}$, but remarked that ``proving such a result to be true is probably very difficult'' \cite{edelman1995complete}. Trefethen also observed this behavior and made the following conjecture, first at the 1995 International Congress on Industrial and Applied Mathematics, and then in his 1997 text {\it Numerical Linear Algebra}, coauthored with David Bau \cite[pg. 169]{trefethen1997numerical}.

\begin{conjecture}[Trefethen, {\cite[pg. 169]{trefethen1997numerical}}]\label{conj:trefethen}
Let $G \in \R^{n \times n}$ be a Gaussian matrix and $\Pi(\cdot)$ be the partial pivoting map. For any $\alpha, M >0$,
\[\P \left[\growth(\Pi(G)) > n^{\frac12 + \alpha}\right] < n^{-M} \]
for all $n$ sufficiently large.
\end{conjecture}

Trefethen later offered \$1000 for a resolution of Conjecture \ref{conj:trefethen} \cite{trefethen2012smart}. 
The numerical evidence for this conjecture is overwhelming; see \cite[Figures 22.1 \& 22.2]{trefethen1997numerical} and their improvement in \cite{driscoll2007searching}. The mathematical progress has been made on two fronts. First, Trefethen himself claimed an unpublished proof of the needed bound for the last pivot $|U_{nn}|$ \cite{trefethen2012smart}, which was released 14 years later, in parallel with this manuscript \cite{trefethen2026lu}. Second, Huang and Tikhomirov proved that a Gaussian matrix under partial pivoting  has growth at most $n^c$ with high probability, for a large unspecified constant $c$ \cite{huang2024average}. 

Another approach aimed at the stability of partial pivoting is smoothed analysis, which considers the performance of an algorithm under small random perturbations to the input. Sankar proved that any matrix plus a small Gaussian perturbation, $A + \epsilon G$ where $\|A\|_2 \le 1$, has growth factor under partial pivoting bounded by $(n/\epsilon)^{O(\log n)}$ with high probability \cite{sankar2004smoothed}. For more information regarding Gaussian elimination with partial pivoting, we refer the reader to \cite{trefethen1997numerical,urschel2025numerical}.

\subsection{Alternatives to Partial Pivoting} Complete pivoting has a growth factor which is only quasi-polynomial; see \cite{bisain2025new,edelman2024some,shah2026entry,wilkinson1961error} for details. However, due to the additional computational complexity required, partial pivoting remains preferred in practice. Demmel, Grigori, and Rusciano considered the benefits of avoiding pivoting altogether, by preconditioning a matrix via pre- and post-multiplication by a Haar orthogonal matrix. They proved that, for any $A \in\GL_n(\R)$, the expected log growth factor of $Q_1 A Q_2$, for independent $Q_1,Q_2 \sim \mathrm{Haar}(\mathrm{O}(n))$, is  $O(\log n)$ \cite[Corollary 5.19]{demmel2023improved}. Butterfly matrices have also been considered as a potential preconditioner to limit growth \cite{parker1995random,peca2023growth} and have the benefit of being easier to apply. More generally, the possibility of avoiding pivoting or preconditioning altogether has been considered for certain random matrix models. For instance, Yeung and Chan proved probabilistic bounds on $\|U\|_{\infty}/\|A\|_{\infty}$ and $\|L\|_{\infty}$ for a Gaussian matrix without pivoting \cite{yeung1997probabilistic}. This analysis was tightened and generalized to the smoothed analysis setting by Sankar, Spielman, and Teng \cite{sankar2006smoothed}. In particular, for Gaussian elimination without pivoting $A+ \epsilon G = LU$, with $\|A\|_2 \le 1$ and $\epsilon \in (0,1)$, they proved that
\[ \P\left[\|L\|_{\infty} > x \right] \lesssim  \frac{n^2}{x} \left( \frac{1}{\epsilon} + \sqrt{\log n} \right) \qquad \text{and} \qquad \P\left[\frac{\|U\|_{\infty}}{\|A+\epsilon G\|_\infty} > 1 + x \right] \lesssim \frac{1}{x} \min \left\{\frac{n^2}{\epsilon}, n^{\frac32} + \frac{n^{\frac12}}{\epsilon^2}\right\}  .\]
Somewhat remarkably, even matrices with entries sampled independently from a non-degenerate discrete distribution with finite support have a polynomial growth factor with constant probability \cite{mateo2026lu}.

\subsection{Our Contributions} Here we prove a number of results regarding the growth factor of random matrices, both without pivoting and with partial pivoting. First, we prove an upper bound on the growth factor of a matrix left-multiplied by a Haar orthogonal matrix without pivoting.

\begin{theorem}\label{thm:orthog}
Let $A \in \GL_n(\mathbb{R})$ and $Q \sim \mathrm{Haar}(\mathrm{O}(n))$. Then
\[\P\left[ \growth(QA) >  \frac{14 \, n^{\frac32} \sqrt{\log(e n)}}{\delta}\right] < \delta \qquad \text{for all} \quad \delta \in (0,1).\]
\end{theorem}

The $\frac32$ exponent in the previous theorem is tight for any non-adaptive orthogonal preconditioner, illustrated by the asymptotic distribution of the growth factor of a Gaussian matrix without pivoting.

\begin{theorem}\label{thm:gauss_nopivot}
Let $G \in \R^{n \times n}$ be a Gaussian matrix. Then $n^{-\frac32}\growth(G) \overset{d}{\rightarrow} X$, where $X \sim \mathrm{InvExp}\big(\frac2{3\sqrt{\pi}}\big)$, i.e., $\Pr[X \le x] = e^{-\frac{2}{3 \sqrt{\pi} x}}$ for $x>0$.\end{theorem}

\begin{figure}[t!]
\centering
  \begin{subfigure}[t]{0.49\textwidth}
    \centering
    \includegraphics[width =.9\textwidth]{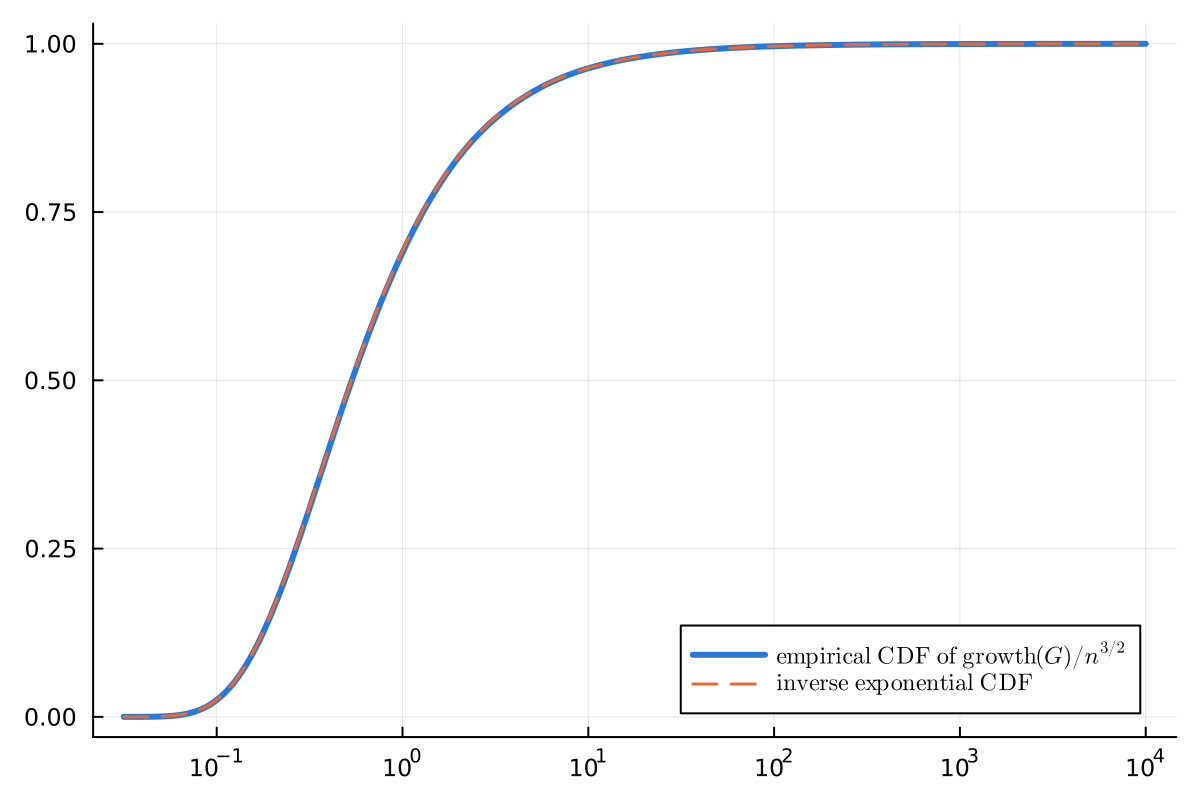}
    \caption{}
  \end{subfigure}\hfill
  \begin{subfigure}[t]{0.49\textwidth}
    \centering
\includegraphics[width =.9\textwidth]{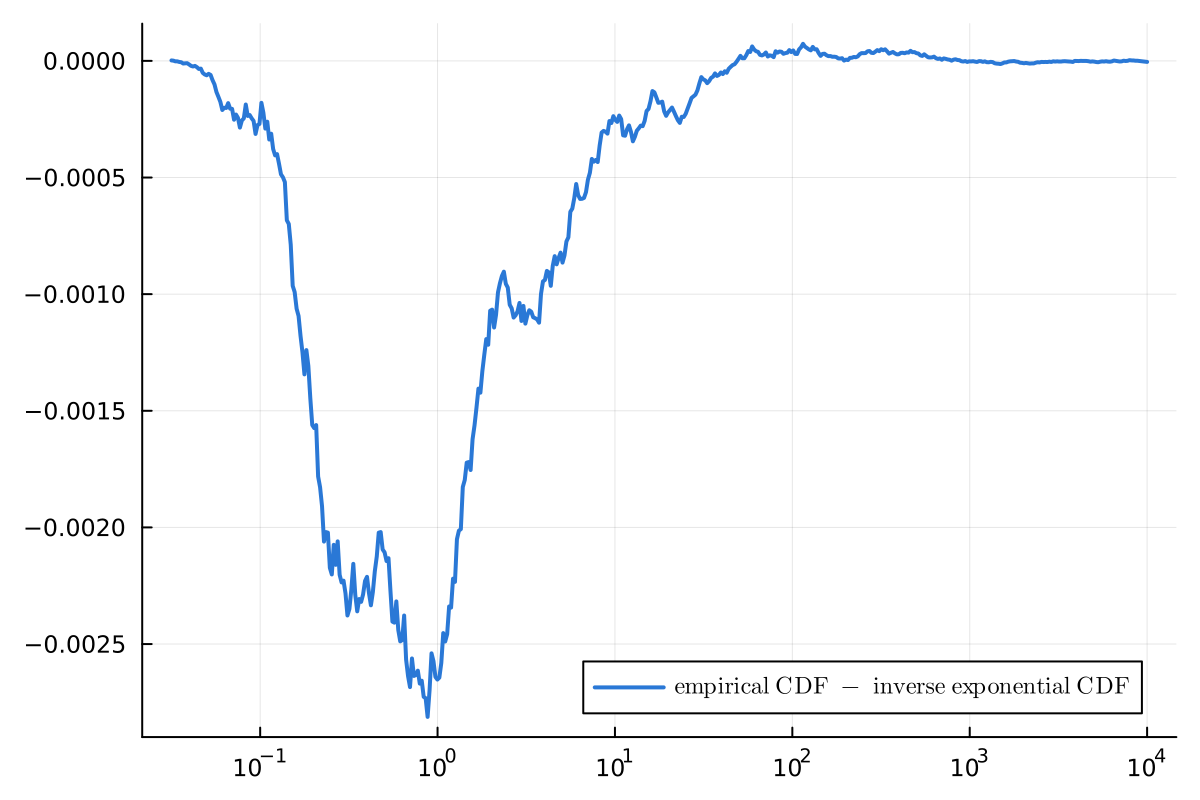}
    \caption{}
  \end{subfigure}
 \caption{We compare the empirical distribution of one million samples of $n^{-\frac32} \growth(G)$, the scaled growth factor of a Gaussian matrix without pivoting, for $n = 64$, to the distribution of an inverse exponential with rate $2/(3\sqrt{\pi})$. In Figure (A), we compare their distributions directly, which are indistinguishable to the human eye. In Figure (B), we plot the difference. The KS distance is $2.8\times 10^{-3}$, illustrating the rapid practical convergence that occurs even for small $n$. Unlike the scaled growth factor of a Gaussian matrix under partial pivoting, $n^{-\frac12}\growth(\Pi(G))$, which we prove is $n^{o(1)}$ with overwhelming probability, $n^{-\frac32} \growth(G)$ is $n^{\Theta(1)}$ with probability $n^{-O(1)}$. The empirical difference in tail distribution between no pivoting and partial pivoting is made readily apparent by comparing the above Figure (A) to Figure \ref{fig:partial_numerics}}
\label{fig:cdfs}
\end{figure}

Numerical experiments suggest that this convergence occurs quite rapidly, even for small values of $n$; see Figure \ref{fig:cdfs}. While Theorems \ref{thm:orthog} and \ref{thm:gauss_nopivot} are fairly straightforward to prove, the results themselves provide further context for the situation regarding partial pivoting, the main focus of this work. In particular, we note that orthogonal preconditioning, a reasonable alternative to partial pivoting, produces a larger growth factor than one expects to see with partial pivoting when applied to a Gaussian matrix. Furthermore, the tails of the bounds decay inverse polynomially, which is slower than one expects for partial pivoting. The following theorem, our main result, proves that the growth factor of a Gaussian matrix is at most $n^{\frac12 + o(1)}$ with overwhelming probability, that is, $1- n^{-\alpha}$ for any $\alpha$.

\begin{theorem}\label{thm:gauss_pp}
For every $\alpha >0$, there exist constants $C_\alpha,n_\alpha$ such that
\[\P\left[\growth(\Pi( G)) > \sqrt{n} \, e^{C_\alpha \sqrt{\log n}}\right] < n^{-\alpha}\qquad \text{for all} \quad n \ge n_\alpha,\]
where $G \in \R^{n \times n}$ is a Gaussian matrix and $\Pi(\cdot)$ is the partial pivoting map.
\end{theorem}

Theorem \ref{thm:gauss_pp} implies Trefethen's Conjecture \ref{conj:trefethen}, as $e^{\sqrt{\log n}} = o(n^c)$ for all $c>0$. The $\frac12$ exponent in the theorem is correct, as $\log \|G\|_{\max} = o(\log n)$ and, by \cite{nguyen2014random}, $\log \| U\|_{\max} \ge \log (|\det(G)|^{\frac1n}) \ge \frac12 \log n - O(1)$ with probability $1-o(1)$. This, combined with Theorem \ref{thm:gauss_pp}, implies the following.

\begin{corollary}
Let $G \in \R^{n\times n}$ be a Gaussian matrix and $\Pi(\cdot)$ be the partial pivoting map. Then
\[\frac{\log \growth(\Pi(G))}{\log n} \overset{p}{\longrightarrow} \frac12.\]
\end{corollary}
Numerical experiments suggest some some further refinements. While the evidence is not overwhelming, empirically, the fluctuations of $n^{-\frac12} \growth(G)$, properly shifted and normalized, seem to resemble Gumbel distributions. This would not be too surprising, as these distributions often appear when taking the maximum of some number of random variables. Somewhat somewhat surprisingly, as $n$ increases, the empirical density of $n^{-\frac12} \growth(G)$ appears to narrow, leaving open the possibility that $n^{-\frac12} \growth(G)$ converges to a constant. See Figure \ref{fig:partial_numerics} for details.

The proof of Theorem \ref{thm:gauss_pp}, suitably adjusted, also provides a smoothed analysis of partial pivoting.

\begin{figure}[t!]
\centering
  \begin{subfigure}[t]{0.49\textwidth}
    \centering
    \includegraphics[width =.9\textwidth]{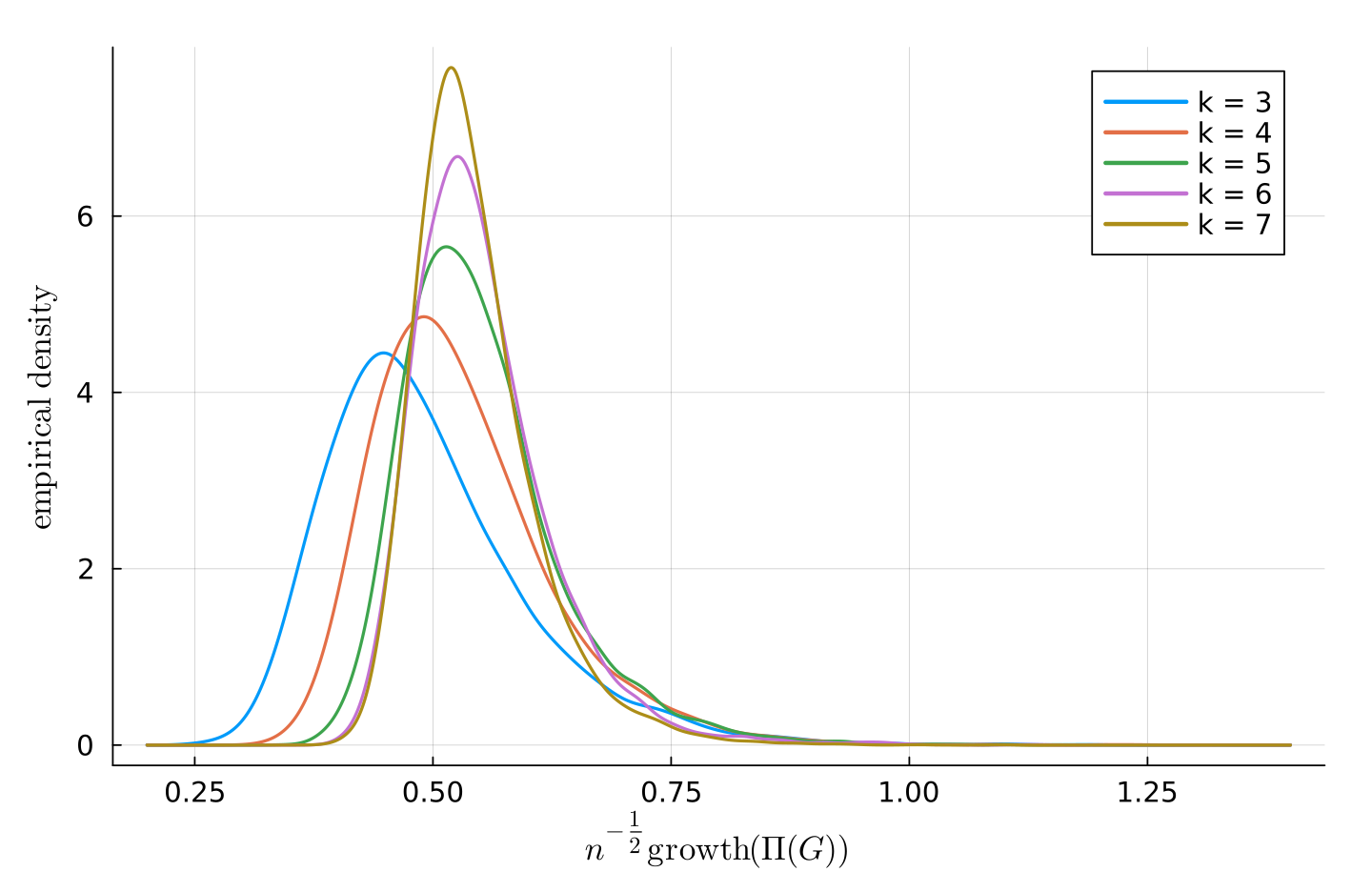}\vspace{-1mm}
    \caption{}
  \end{subfigure}\hfill
  \begin{subfigure}[t]{0.49\textwidth}
    \centering
\includegraphics[width =.9\textwidth]{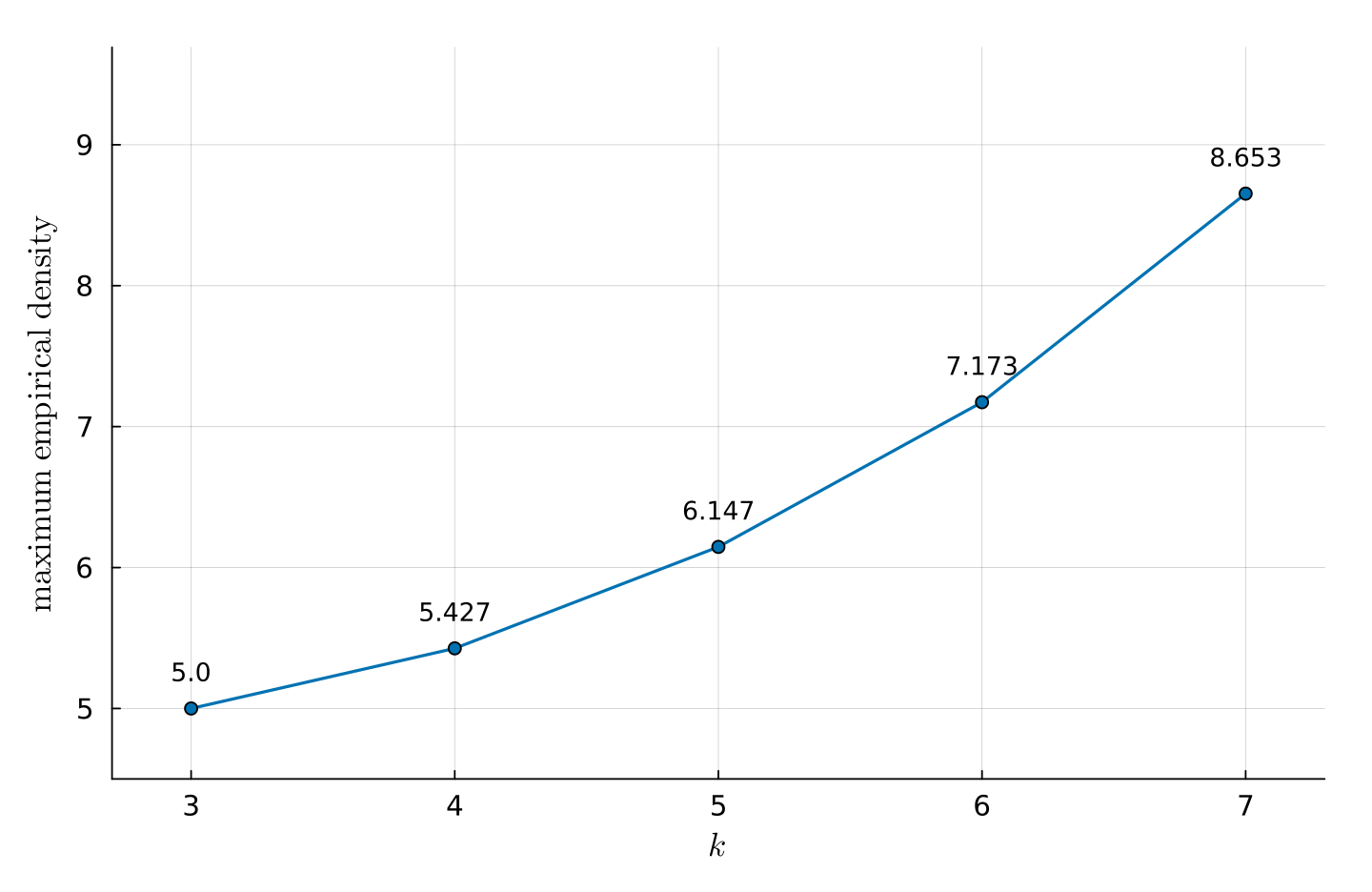}\vspace{-1mm}
    \caption{}
  \end{subfigure} \\
    \begin{subfigure}[t]{0.49\textwidth}
    \centering
    \includegraphics[width =.9\textwidth]{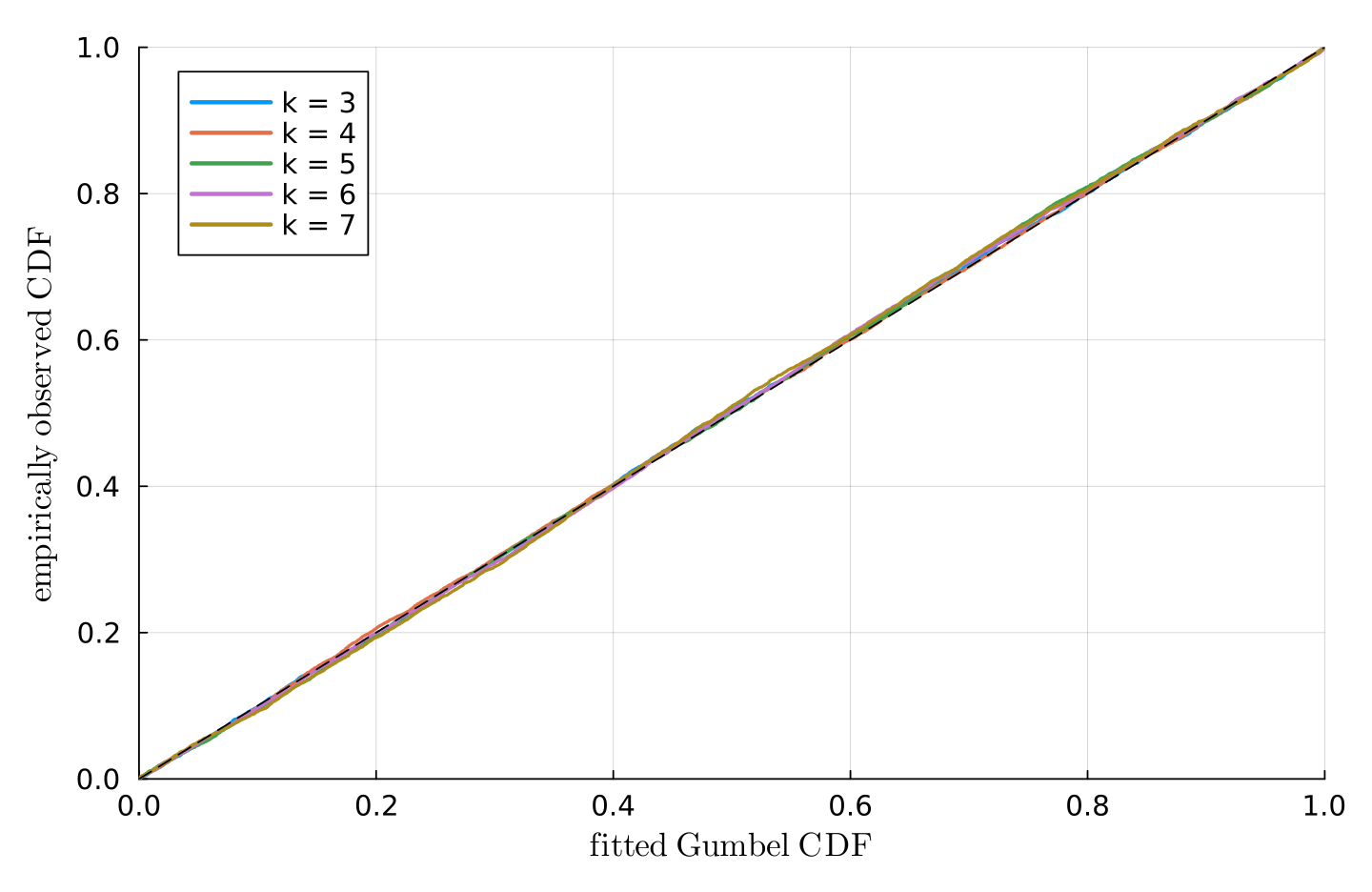}\vspace{-1mm}
    \caption{}
  \end{subfigure}\hfill
  \begin{subfigure}[t]{0.49\textwidth}
    \centering
\includegraphics[width =.9\textwidth]{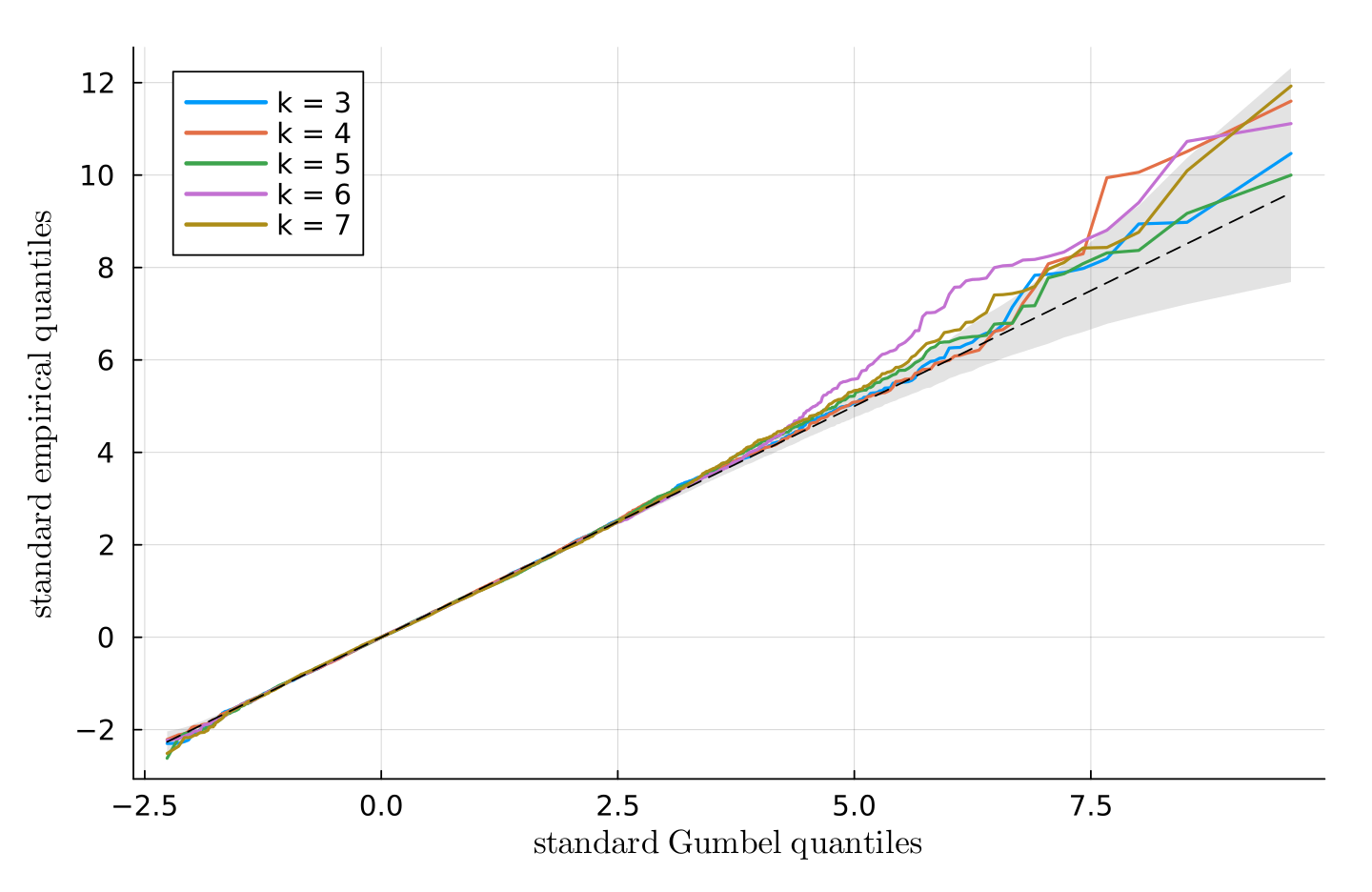}\vspace{-1mm}
    \caption{}
  \end{subfigure}
 \caption{We compute $7500$ samples of $n^{-\frac12}\growth(\Pi(G))$ for $n = 4^k$, $k = 3,\ldots,7$. In Figure (A), we plot their empirical densities. Note that the tails decay rapidly and the densities appear to narrow somewhat. In Figure (B), for each value of $k$, we plot the maximum empirical density in a region of length $1/100$, further illustrating this narrowing. The empirical densities look somewhat similar to Gumbel distributions. We fit each value of $k$ individually, and, after proper shifting and scaling, compare them to their fitted counterparts. In Figures (C) and (D), we provide a P-P plot and Q-Q plot, respectively. The grey region in Figure (D) notes the area that the empirical quantiles of $7500$ Gumbell samples would lie in, with $95\%$ likelihood.}
\label{fig:partial_numerics}
\end{figure}

\begin{theorem}\label{thm:smooth_pp}
For every $\alpha >0$, there exist constants $C_\alpha, n_\alpha$ such that, for all $n \ge n_\alpha$, every matrix $A \in \R^{n\times n}$ with $\|A\|_2 \le 1$ and every $\epsilon \in (0,1)$ satisfy
\[\P\left[ \growth(\Pi(A+\epsilon G)) >  n^{C_\alpha}\epsilon^{-\frac{2}{\log2} \log \log n} \right] < n^{-\alpha},\]
where $G \in \R^{n \times n}$ is a Gaussian matrix and $\Pi(\cdot)$ is the partial pivoting map.
\end{theorem}

This shows that polynomially small perturbations produce $n^{O(\log \log n)}$ growth. While the correct behavior is likely $n^{O(1)}$, this is a significant improvement over the longstanding $n^{O(\log n)}$ bound.

\subsection{Remainder of Paper} In Section \ref{sec:orthogonal}, we prove that a Haar orthogonally preconditioned matrix rarely has growth much larger than $n^{\frac32}$ (Theorem \ref{thm:orthog}). In Section \ref{sec:gauss}, we prove that the distribution of the growth of a Gaussian matrix without pivoting, divided by $n^{\frac32}$, converges to an inverse exponential distribution (Theorem \ref{thm:gauss_nopivot}). In Section \ref{sec:technical}, we prove six technical lemmas regarding Gaussian matrices and singular values, which we use in Sections \ref{sec:partial} and \ref{sec:smooth}. In Section \ref{sec:partial}, we prove that the growth factor of a Gaussian matrix under partial pivoting is rarely much more than $n^{\frac12}$ (Theorem \ref{thm:gauss_pp}). In Section \ref{sec:smooth}, we prove that any matrix, of norm at most one, plus an $\epsilon$ sized Gaussian perturbation rarely has growth much larger than $n^{O(1)} \epsilon^{-O(\log \log n)}$ (Theorem \ref{thm:smooth_pp}). The proof of Theorem \ref{thm:smooth_pp} is a consequence of a more general version of Theorem \ref{thm:gauss_pp} (Theorem \ref{thm:smooth_general}). However, for clarity, we present the slightly simpler proof of the Gaussian case first in Section \ref{sec:partial}. Then, in Section \ref{sec:smooth}, we present a generalization of the same argument, with far greater brevity.

\section{Growth of Orthogonally Preconditioned Matrices without Pivoting}\label{sec:orthogonal}

Here we prove that the growth factor of $QA$, where $A \in \mathrm{GL}_n(\R)$ is fixed and $Q \sim \mathrm{Haar}(\mathrm{O}(n))$, is usually not much bigger than $n^{\frac32}$ (Theorem \ref{thm:orthog}).  Our main technical tool is the following lemma regarding the inverse of a submatrix of a Haar orthogonal matrix.

\begin{lemma}\label{lm:orthog}
Let $Q \sim \mathrm{Haar}(\mathrm{O}(n))$, $S,T\subset[n]$ be fixed subsets of size $m$, and $\bm x \in \mathbb{S}^{m-1}(\R)$ and $\bm y_1, \bm y_2, \ldots, \bm y_k \in \mathbb{S}^{m-1}(\R)$ be fixed for some $k \le m$. Then
\[\P \left[ \max_{j \in [k]} |\bm y_j^T (Q_{S,T})^{-1} \bm x | >t \right] \le \frac{6
}{t} \sqrt{\frac{(n-m+1)\log(ek)}{m}} \qquad \text{for all} \quad t >0.
\]
\end{lemma}

\begin{proof}
Without loss of generality, we may take $B = Q_{[m],[m]}$ and $\bm x = \bm e_1$. First we identify the distribution of $(Q_{[m],[m]})^{-1} \bm{e}_1$. Note that $\bm u := Q_{:,[m]}(Q_{[m],[m]})^{-1} \bm{e}_1$ satisfies $\bm u_1 = 1$ and is in the intersection of the column space of $Q_{:,[m]}$ and $\mathrm{span}\{\bm e_1, \bm e_{m+1}, \ldots, \bm e_n\}$, which is almost surely one-dimensional. A unit vector in this intersection, after multiplication by a random $\pm 1$ sign, is distributed uniformly on $\mathbb{S}^{n-m}(\R)$, the unit sphere in this $(n-m+1)$ dimensional subspace, and so 
\[\|(Q_{[m],[m]})^{-1} \bm{e}_1\|_2 = \|\bm u\|_2 \overset{d}{=}|\bm \eta_1|^{-1},\]
where $\bm \eta \sim \mathrm{Unif}(\mathbb{S}^{n-m}(\R))$. By the orthogonal invariance of $Q_{[m],[m]}$, the direction of $(Q_{[m],[m]})^{-1} \bm{e}_1$ is uniform and independent of its length. Therefore, $(Q_{[m],[m]})^{-1} \bm{e}_1 \overset{d}{=} |\bm \eta_1|^{-1} \bm \xi $, where $\bm \xi \sim \mathrm{Unif}(\mathbb{S}^{m-1}(\R))$ is independent of $\bm \eta$. By \cite[Lemma 2.2]{dasgupta2003elementary},
\[ \P\left[|\bm \eta_1 |<a \right] \le 2 \sqrt{n-m+1} \; a\qquad \text{and} \qquad \E\left[ \max_{j \in [k]} |\bm y_j^T \bm \xi|  \right] \le 3 \sqrt{\frac{\log(ek)}m},\]
where the second estimate follows from a union bound and integration of their upper tail estimate. By independence,
\begin{align*}
\P \left[ \max_{j \in [k]} |\bm y_j^T (Q_{[m],[m]})^{-1} \bm x | >t \right] &= \E \left[ \P \left[ |\bm \eta_1| < \frac{1}{t} \max_{j \in [k]} |\bm y_j^T \bm \xi| \, \bigg| \max_{j \in [k]} |\bm y_j^T \bm \xi| \right] \right] \\
&\le \frac{2 \sqrt{n-m+1}}{t} \E\left[ \max_{j \in [k]} |\bm y_j^T \bm \xi|  \right] \\
&\le \frac{6}{t} \sqrt{\frac{(n-m+1) \log(ek)}m}.
\end{align*}
\end{proof}

Repeated application of Lemma \ref{lm:orthog} is sufficient to prove our desired Theorem \ref{thm:orthog}.

\begin{proof}[Proof of Theorem \ref{thm:orthog}]
Without loss of generality, by $Q \sim \mathrm{Haar}(\mathrm{O}(n))$ and the QR factorization, it suffices to consider upper triangular $A$. Then, by the block inversion formula,
\[(QA)/(QA)_{[k],[k]} = (Q/Q_{[k],[k]})A_{[k]^c,[k]^c} = (Q_{[k]^c,[k]^c})^{-T}A_{[k]^c,[k]^c},\]
where $[k]^c= [n]\setminus[k]$. Let $\alpha = \max_j \|A \bm e_j\|_2 \le \sqrt{n} \|QA\|_{\max}$. Then every column of $\alpha^{-1} A$ has two-norm at most one. Applying Lemma \ref{lm:orthog} to each row of $\alpha^{-1} U$ and summing the failure probabilities, we obtain
\[\P\left[ \frac{\|U\|_{\max}}{\|QA\|_{\max}} >t \right] \le  \frac{6\sqrt{n}}{t} \sum_{m = 1}^n \sqrt{\frac{(n-m+1) \log (em)}{m}} \le \frac{6 n \sqrt{\log(en)}}{t} \sum_{m=1}^n \frac1{\sqrt{m}} \le \frac{12 n^{3/2} \sqrt{\log(en)}}{t} .\]
Next, we consider $L$. Note that, because $A$ is upper triangular, $QA$ and $Q$ have the same lower triangular factor \cite[proof of Lemma 5.1]{shah2026entry}. The subdiagonal entries of the $i^{th}$ column of $L$ have law $(g_{i+1},\ldots, g_{n})/g_{i}$ for $g_i,\ldots,g_n$ i.i.d. $N(0,1)$, and therefore standard Gaussian estimates give, for $t \ge 1$,
\begin{equation}\label{gauss_L_law}
\P\left[\|L\|_{\max} >t\right] \le \sum_{i = 1}^{n-1} \P\left[\max_{j>i}|L_{ji}| >t \right] \le \frac{n-1}{t}\sqrt{\frac2 {\pi}}\, \E \left[ \max_{j \in [n]} |g_j|\right] \le \frac{2(n-1) \sqrt{\log (en)}}{\sqrt{\pi} t}.
\end{equation}
Combining the two probabilistic bounds for $L$ and $U$ completes the proof.
\end{proof}

\section{Growth of Gaussian Matrices without Pivoting}\label{sec:gauss}

Here we prove that the asymptotic distribution of the growth factor of a Gaussian matrix without pivoting, divided by $n^{\frac32}$, is inverse exponential. We require two lemmas. The first concerns the distribution of the largest entry in a row of $U$, which is related to beta random variables. The second concerns the maximum of the inverse square roots of independent beta random variables.

\begin{lemma}\label{lm:row_norms}
Let $G$ be an $n \times n$ Gaussian matrix with LU factorization $G = LU$, and
\[X_k = \| (-G_{k+1,[k]}G_{[k],[k]}^{-1},1)\|_2 \qquad \text{and} \qquad Y_k = \max_{j \ge k+1} \left| \frac{(-G_{k+1,[k]}G_{[k],[k]}^{-1},1)G_{[k+1],j}}{X_k} \right|\]
for $k \in [n-1]$. Then
\[\frac{1}{X_k^2} \sim \mathrm{Beta}\left(\frac12,\frac k2\right), \qquad Y_k \overset{d}{=} \max_{j \in [n-k]} |g_j|, \qquad \text{and} \qquad \max_{j\ge k+1} |U_{k+1,j}| = X_k Y_k\]
for all $k \in [n-1]$, where $g_1,\ldots,g_{n-1}$ are i.i.d. $N(0,1)$. Furthermore, $X_1,\ldots,X_{n-1}$ are mutually independent, and $X_k$ is independent of $Y_k$ for $k \in [n-1]$.
\end{lemma}

\begin{proof}
Let $\bm x_k^T = (-G_{k+1,[k]}G_{[k],[k]}^{-1},1)$ for $k \in [n-1]$. We have
\[U_{k+1,j} = (G/G_{[k],[k]})_{1,j-k} =\bm x_k^T G_{[k+1],j} \qquad \text{for} \quad k+1 \le j,\]
from which $\max_{j \ge k+1} |U_{k+1,j}| = X_k Y_k$ for $k \in [n-1]$ follows immediately. Conditional on the first $k$ columns, $\bm x_k^T/X_k$ is fixed, the remaining columns are Gaussian, and so their inner products with $\bm x_k^T/X_k$ are independent $N(0,1)$. This gives the distribution of $Y_k$ and its independence from $X_k$.

Next, consider $X_k$. The vector $\bm{x}_k/X_k$ spans $\mathrm{null}(G_{[k+1],[k]}^T)$, and is therefore, after a random signing, uniformly distributed on the unit sphere $\mathbb{S}^{k}(\mathbb{R})$. $X_k^{-2}$ is the squared magnitude of its last entry, and so $X_k^{-2}$ is Beta distributed with parameters $\alpha = \frac{1}{2}$ and $\beta = \frac{k}{2}$. Finally, the independence of $X_k$ to $X_{k+1},\ldots,X_{n-1}$ follows from considering the application of a Haar orthogonal matrix to the first $k+1$ rows of $G$. This operation preserves the Gaussian distribution of the matrix, fully randomizes $X_k$, but leaves the values of $X_{k+1},\ldots,X_{n-1}$ unchanged.
\end{proof}

\begin{lemma}\label{lm:beta}
Let $X_1,\ldots, X_{n-1}$ be independent positive random variables satisfying $X_k^{-2} \sim \mathrm{Beta}(\frac12,\frac k2)$ for $k \in [n-1]$, 
\[c_k = \frac{2 \Gamma(\frac{k+1}{2})}{\sqrt{\pi} \Gamma(\frac k2)} \quad \text{for} \quad k \in [n-1], \qquad \text{and} \qquad \alpha_n =  \sum_{k=1}^{n-1} c_k.\]
Then $\alpha_n = \displaystyle{\frac{2 \sqrt{2}}{3 \sqrt{\pi}} n^{\frac32} }+ O(n^{\frac12})$ and, for fixed $x>0$,
\[\P\left[\max_{k \in [n-1]} X_k \le \alpha_n x \right] \rightarrow e^{-\frac{1}{x}}.\]
\end{lemma}

\begin{proof}
Using the estimate $\Gamma(x+c)/\Gamma(x) = x^c + O_c(x^{c-1})$, we have
\[\alpha_n = \frac{2}{\sqrt{\pi}} \sum_{k=1}^{n-1} \frac{ \Gamma(\frac{k+1}{2})}{ \Gamma(\frac k2)} =  \sqrt{\frac2\pi} \, \sum_{k=1}^{n-1} k^{\frac12} + O(n^{\frac12}) = \frac{2 \sqrt{2}}{3 \sqrt{\pi}} n^{\frac32} + O(n^{\frac12}).\]
For $n$ sufficiently large, $1/(\alpha_n x) < 1/2$ and $|(1-u^2)^{\frac{k}{2}-1}-1|\le C k u^2$ for all $k \in [n-1]$ and $0 \le u \le 1/(\alpha_n x)$, for some constant $C$. Therefore,
\[p_k:=\P\left[X_k \ge \alpha_n x\right] = c_k  \int_{0}^{\frac1{\alpha_nx}} (1-u^2)^{\frac k2 -1} \, d u = \frac{c_k}{\alpha_n x} \left(1 + O_x\left(n^{-2}\right)\right).\]
Independence and the estimate $\log(1-p) = -p+O(p^2)$ imply our desired result
\[\log \P\left[\max_{k \in [n-1]} X_k \le \alpha_n x \right] = - \sum_{k=1}^{n-1} p_k + O\!\left(\sum_{k=1}^{n-1} p_k^2 \right) = - \frac1x + O_{x}\left(n^{-1}\right).\]
\end{proof}

Using Lemmas \ref{lm:row_norms} and \ref{lm:beta}, we are now prepared to prove Theorem \ref{thm:gauss_nopivot}.

\begin{proof}[Proof of Theorem \ref{thm:gauss_nopivot}]
Lemma \ref{lm:beta} gives control of $X_k$. In order to estimate $\|U\|_{\max}$, we require concentration of $Y_k$. Let $g_1,g_2,\ldots$ be i.i.d. $N(0,1)$. We have
\[\E\left[\frac{\max_{j \in [m]} |g_j|}{\sqrt{2 \log m}} \right] \rightarrow 1 \qquad \text{and} \qquad \E\left[\left|\max_{j \in [m]} |g_j| - \E \left[ \max_{j \in [m]} |g_j| \right] \right| \right] \le \sqrt{2\pi}, \]
by \cite[Exercise 2.11c]{wainwright2019high} and \cite[Proposition 10.3]{halko2011finding}, respectively. The quantity $\E \max_{j \in [m]} |g_j|$ is monotonic in $m$, and so, for every fixed $\epsilon \in (0,1)$,
\[\sup_{\epsilon n \le m \le n} \E \left[ \left| \frac{\max_{j \in [m]} |g_j|}{\sqrt{2 \log n}} -1\right|\right]\rightarrow 0  \qquad \text{and}  \qquad \sup_{m \in [n]} \E \left[ \left| \frac{\max_{j \in [m]} |g_j|}{\sqrt{2 \log n}} -1\right|\right]\le C.\]
We have $X_k^{-2} \sim \mathrm{Beta}(\frac12,\frac k2)$, and so $\P[X_k > t] \le \frac{\pi c_k}{2t}$ for all $t>0$. Then, by conditioning on each $Y_k$ and using the independence of $X_k$ and $Y_k$,
\[\P\left[\max_{k \in [n-1]} X_k \left|\frac{Y_k}{\sqrt{2 
\log n}} - 1 \right| > \delta \alpha_n \right]  \le \frac{\pi}{2\delta \alpha_n} \sum_{k =1}^{n-1} c_k \, \E \left[\left|\frac{Y_k}{\sqrt{2 
\log n}} - 1 \right| \right] \rightarrow 0\]
for all $\delta >0$. Note that the first row of $U$ is the first row of $G$, and can be safely ignored when considering the scale $\alpha_n \sqrt{2 \log n}$. Therefore, by Lemma \ref{lm:beta},
\[\P\left[\frac{ \|U\|_{\max}}{\alpha_n \sqrt{2 \log n}} \le  x \right] \rightarrow e^{-\frac{1}{x}}.\]
Now, consider $L$. The subdiagonal entries of the $i^{th}$ column of $L$ have the same law as in the proof of Theorem \ref{thm:orthog}, and so \eqref{gauss_L_law} holds here as well and $n^{-3/2}\|L\|_{\max} \rightarrow 0$ in probability. 

Finally, we have $\|G\|_{\max} /(2 \sqrt{\log n} ) \rightarrow 1$ in probability, and so, by Slutsky's theorem,
\[\frac{\|U\|_{\max}}{n^{3/2} \|G\|_{\max}} =  \frac{ \alpha_n}{\sqrt{2} n^{3/2}}\frac{\|U\|_{\max}}{\alpha_n \sqrt{2 \log n}} \frac{2 \sqrt{\log n}}{\|G\|_{\max}} \overset{d}{\longrightarrow}  X.\]
\end{proof}

\section{Singular Values and Gaussian Matrices}\label{sec:technical}

Here, we recall and prove six technical results regarding singular values and Gaussian matrices required for our analysis of partial pivoting in Sections \ref{sec:partial} and \ref{sec:smooth}. For Section \ref{sec:partial}, the analysis of a Gaussian matrix under partial pivoting, we apply the below lemmas with $\Delta = 0$. We denote the singular values of a matrix $M \in \R^{m \times n}$ by $\sigma_1(M) \ge \ldots \ge \sigma_{\min(m,n)}(M)$, and, for a square $M \in \R^{n \times n}$, let
\[\Sigma_k(M) = \sum_{i=1}^{n-k} \frac1{\sigma_i(M)^{2}}\]
denote the squared Frobenius norm of the inverse of a matrix after truncating the $k$ largest singular values of the inverse. By convention, we set $\Sigma_k(M) = 0$ when $k \ge n$. Let $M^{\dagger}$ denote the Moore-Penrose pseudoinverse of a matrix $M \in \R^{m \times n}$.
\begin{lemma}\label{lm:gauss_bounds}
Let $G \in \R^{m \times n}$ be Gaussian and $M \in \R^{n \times p}$, $\Delta \in \R^{m\times n}$ be fixed, with $\|\Delta\|_2 \le K$. Then
\begin{align}
\P\left[\|(G+\Delta)M\|_F^2 > 2 (\sqrt{m}+K)^2\|M\|_F^2 + 4 x \|M\|_2^2\right] &\le e^{-x} \qquad \text{and} \\
\P\left[\sigma_k((G+\Delta)M) > \|M\|_F + (\sqrt{m} +K + \sqrt{2 x/k}) \|M\|_2\right] &\le e^{-x}
\end{align}
for all $x >0$ and $k \le \min(m,p)$.
\end{lemma}

\begin{proof}
Let $f_k(A):= \sqrt{\sum_{i =1}^k \sigma_i(A)^2}$. Both $G \mapsto \|(G+\Delta)M\|_F$ and $G \mapsto f_k((G+\Delta)M)$ are $\|M\|_2$-Lipschitz, and so, by \cite[Proposition 10.1]{halko2011finding} and the triangle inequality,
\[\E\,\|(G+\Delta)M\|_F \le \left( \sqrt{m} +K \right)\|M\|_F \qquad \text{and} \qquad \E\, f_k((G+\Delta)M)\le \sqrt{k} \left( \|M\|_F +  \left( \sqrt{m} +K \right)\|M\|_2\right).\]
The desired result follows from Gaussian concentration \cite[Proposition 10.3]{halko2011finding} with deviation $\sqrt{2x} \|M\|_2$ and the bounds $(a+b)^2 \le 2(a^2+b^2)$ and $\sigma_k(A) \le f_k(A)/\sqrt{k}$. \end{proof}

\begin{lemma}\label{lm:wide_gauss}
Let $G \in \R^{m\times n}$ be Gaussian and $\Delta \in \R^{m \times n}$ be fixed, with $m \le r$ and $n \ge 3 r$. Then
\[\P\left[ \|(G +\Delta)^\dagger \|_2^2 > \frac{2 e^{\frac{2x}{r}+2}}{r}  \right] \le e^{-x}\]
for all $x \ge 1$.
\end{lemma}

\begin{proof}
Let $Q \in \R^{n \times 2r}$ have orthonormal columns in $\mathrm{null}(\Delta)$. Then $(G+\Delta)Q \in \R^{m \times 2r}$ is Gaussian and $(G+\Delta)(G+\Delta)^T \succeq (G+\Delta)QQ^T(G+\Delta)^T$, so $\|(G +\Delta)^\dagger \|_2 \le \|((G +\Delta)Q)^\dagger \|_2 = \|(GQ)^\dagger\|_2$. Next, add $r-m$ Gaussian rows to $GQ$, which can only decrease the smallest singular value. When $r>1$, apply \cite[Proposition A.3]{halko2011finding} with $t = e\sqrt{2/r}\,e^{x/r}$ to the resulting $r \times 2r$ matrix. When $r=1$, this is a standard result about $\chi_2^2$. 
\end{proof}

\begin{lemma}\label{lm:gauss_sing_val}
Let $G \in \R^{n\times n}$ be Gaussian and $\Delta \in \R^{n \times n}$ be fixed. Then
\[\P\left[ \sigma_{n-j}(G+\Delta) \le \frac{c j \theta}{\sqrt n} \right] \le n^{j+1} \theta^{\frac14 j^2}\]
for all $4 \le j < n$ and $\theta \in (0,1]$, for some absolute constant $c \in (0,1]$.
\end{lemma}

\begin{proof}
This follows immediately from \cite[(7) of Theorem 1.7]{nguyen2018random} by taking $\gamma = 1/2$, $k = j+1$, $\epsilon = c j \theta/(j+1)$, using $\sup_{a}\P[|X-a|\le \epsilon]\le \sqrt{2/\pi} \; \epsilon$ for $X \sim N(0,1)$, and choosing $c$ sufficiently small.
\end{proof}

\begin{lemma}\label{lm:gauss_convex}
Let $X \in \R^n$ be distributed as $N(\bm v,I_n)$ conditioned on lying in a convex set $\mathcal{K} \subset \R^n$ of positive Gaussian measure with $\mathcal{K}=-\mathcal{K}$. Then
\begin{equation}
\left\| \E X  \right\|_2 \le \| \bm v \|_2 \qquad \text{and} \qquad \E \left[ e^{\bm u^T (X - \E X)}\right] \le e^{\frac12 \|\bm u\|_2^2} \qquad \text{for all} \quad \bm u \in \R^n.
\end{equation}
For fixed $M \in \R^{n \times p}$ and $x >0$,
\begin{equation}
\P\left[\|X^T M\|_2 > \sqrt{ 2 \|M\|_F^2 + 4x \|M\|_2^2} + \| \bm v\|_2\|M\|_2 \right] \le e^{-x}.
\end{equation}
\end{lemma}

\begin{proof}
Let $\varphi_n(\bm x) = (2 \pi)^{-n/2}e^{-\|\bm x\|_2^2/2}$ be the density of a standard Gaussian vector and
\[f(\bm y) = \log \int_{\mathcal{K}} e^{\bm y^T \bm x} \varphi_n(\bm x) \, d \bm x.\]
Then $X_{\bm y}$, distributed as $N(\bm y,I_n)$ conditioned on lying in $\mathcal{K}$, has density $\bm{1}_{\mathcal{K}}(\bm x) e^{\bm y^T \bm x - f(\bm y)} \varphi_n(\bm x)$. In addition, $\nabla f(\bm y) = \E X_{\bm y}$ and $\nabla^2 f(\bm y) = \mathrm{Cov}(X_{\bm y})$. Therefore $\nabla^2 f(\bm y) \succeq 0$. Next, we aim to upper bound this quantity. By \cite[Theorem 6]{prekopa1973logarithmic},
\[\log \int_{\R^n} \bm{1}_{\mathcal{K}}(\bm x) \varphi_n(\bm x - \bm y) \, d \bm x = f(\bm y) - \frac12 \|\bm y\|_2^2\]
is concave, and so $\nabla^2 f(\bm y) \preceq I_n$. By the symmetry of $\mathcal{K}$, $\nabla f( 0) = 0$. We have
\[\E X_{\bm y} = \nabla f(\bm y) - \nabla f(0) = \int_0^1 \nabla^2 f(t \bm y) \bm y \, d t,\]
and, by $0 \preceq \nabla^2 f(t \bm y) \preceq I_n$,
\[\|\E X_{\bm y}\|_2 \le \int_0^1 \|\nabla^2 f(t \bm y) \bm y\|_2 \, dt \le \|\bm y\|_2.\]
Similarly,
\begin{align*}
\log \E \left[ e^{\bm u^T (X_{\bm y} - \E X_{\bm y})}\right] &= f(\bm y + \bm u) - f(\bm y) - \bm u^T \E X_{\bm y} \\ &= \int_0^1 (1-t) \, \bm u^T \nabla^2 f(\bm y +t \bm u) \bm u \, d t \\
&\le \int_0^1 (1-t) \|\bm u\|_2^2 \, dt \\
&= \frac12 \|\bm u\|_2^2.
\end{align*}
By \cite[Theorem 2.1]{hsu2012tail} applied to $X_{\bm y} - \E X_{\bm y}$ with $A = (M M^T)^{1/2}$ and $\mathrm{trace}((MM^T)^2) \le \|M\|_2^2 \|M\|_F^2$,
\[\|(X_{\bm y} - \E X_{\bm y})^T M\|_2^2 \le 2 \|M\|_F^2 + 4x \|M\|_2^2\]
with probability at least $1- e^{-x}$, which implies our final desired bound, by the triangle inequality.
\end{proof}

\begin{lemma}\label{lm:gauss_extension}
Let $k<n$ and the fixed matrix $M \in \R^{n \times n}$ satisfy $\sigma_{n-k}(M) \ge \mu$ and $\Sigma_k(M) \le 2 k \mu^{-2}$ for some $0<\mu \le \sqrt{k}$. Let $G \in \R^{n \times 4k}$ be Gaussian, $\Delta \in \R^{n \times 4k}$ be fixed with $\|\Delta\|_2 \le K$, and $B = [M \;(G + \Delta)]$. Then, for any $x \ge 1$,
\begin{align}
\|B^\dagger\|_2^2 &\le C \mu^{-2} (1+K^2/k)e^{\frac{Cx}{k}} \qquad \text{and}\\
\|B^\dagger\|_F^2 &\le C k \mu^{-2} (1+K^2/k)e^{\frac{Cx}{k}}
\end{align}
with probability at least $1 - 3 e^{-x}$, for some absolute constant $C$.
\end{lemma}

\begin{proof}
Without loss of generality, let $M$ be diagonal with $M_{ii} = \sigma_{n-i+1}(M)$. Let $\hat M = M_{[n]\setminus[k],[n]\setminus[k]}$, partition $G + \Delta = \begin{bmatrix} X \\ Y \end{bmatrix}$, where $X \in \R^{k \times 4k}$ and $Y \in \R^{(n-k) \times 4k}$, and let $Q = X^T(XX^T)^{-1/2} \in \R^{4k \times k}$. Now consider the matrix
\[\hat B = \begin{bmatrix} XQ & 0 \\ Y Q & \hat M \end{bmatrix} \qquad \text{with} \qquad \hat B^{-1} = \begin{bmatrix} (XQ)^{-1} & 0 \\ - \hat M^{-1}  Y Q (XQ)^{-1} & \hat M^{-1}\end{bmatrix}.\]
We have
\[B B^T - \hat B \hat B^T = \begin{bmatrix} M_{[k],[k]}M_{[k],[k]}^T & 0 \\  0 & 0  \end{bmatrix} + (G+\Delta)(I-Q Q^T)(G + \Delta)^T \succeq 0,\]
and so we may instead work with $\hat B$ and bound the blocks of $\hat B^{-1}$. By Lemma \ref{lm:gauss_bounds} applied to $(YQ)^T \hat M^{-1}$, conditional on $X$, and Lemma \ref{lm:wide_gauss} applied to $X$,
\begin{align*}
\|\hat M^{-1}YQ\|_F^2 &\le 2(\sqrt{k}+K)^2 \|\hat M^{-1}\|_F^2 + 4x \|\hat M^{-1}\|_2^2 \le C \mu^{-2} (k^2 + kK^2 + x), \\
\|\hat M^{-1}YQ\|_2^2 &\le \left(\|\hat M^{-1}\|_F +(\sqrt{k}+K+\sqrt{2x})\|\hat M^{-1}\|_2\right)^2 \le C \mu^{-2}(k+K^2 + x), \quad \text{and} \\
     \|(XQ)^{-1}\|_2^2 &= \|X^\dagger\|_2^2 \le \frac{2e^2}{k} e^{2x/k}
\end{align*} 
with probability at least $1-3 e^{-x}$. Combining this with the bounds
\begin{align*}
\|\hat B^{-1}\|_F^2 &\le (k +  \|\hat M^{-1}YQ\|_F^2 )\|(XQ)^{-1}\|_2^2+ 2 k \mu^{-2} \quad \text{and} \\
\|\hat B^{-1}\|_2^2 &\le 3\big[(1+\|\hat M^{-1}YQ\|_2^2)\|(XQ)^{-1}\|_2^2+\mu^{-2}\big],
\end{align*}
and the inequality $1 + x/k + K^2/k \le (1+K^2/k)e^{x/k}$ completes the proof.
\end{proof}

\begin{lemma}\label{lm:adding_rows}
Let $M \in  \R^{n \times (n+k)}$ have rank $n$, the columns of $Q \in \R^{(n+k)\times k}$ be an orthonormal basis for $\mathrm{null}(M)$, $B \in \R^{k \times (n+k)}$, and $j<k/2$. Then
\begin{equation}
\sigma_{n+k-2j}\left(\begin{bmatrix} M \\ B \end{bmatrix}\right) \ge \left(\|M^\dagger\|_2 + \frac{1 + \sigma_{j+1}(BM^\dagger)}{\sigma_{k-j}(BQ)} \right)^{-1}.
\end{equation}
\end{lemma}

\begin{proof}
Let $W \in \R^{k \times (k-2j)}$ have orthonormal columns orthogonal to the bottom $j$ left singular directions of $BQ$ and top $j$ left singular directions of $BM^\dagger$, so that 
\[\sigma_{k-2j}(W^T BQ)\ge \sigma_{k-j}(BQ) \qquad  \text{and} \qquad  \|W^TBM^\dagger\|_2 \le \sigma_{j+1}(BM^{\dagger}).\]
We have
\[\sigma_{n+k-2j}\left(\begin{bmatrix} M \\ B \end{bmatrix}\right) \ge \sigma_{n+k-2j}\left(\begin{bmatrix} M \\ W^T B \end{bmatrix}\right) = \sigma_{\min}\left(\begin{bmatrix} M \\ W^T B \end{bmatrix}\right),\]
and so it suffices to consider the system $M \bm x = \bm a $, $W^T B \bm x = \bm b$ and control the size of the solution $\bm x = M^\dagger \bm a + Q (W^TBQ)^\dagger (\bm b - (W^TBM^\dagger) \bm a)$. We have
\begin{align*}
\|\bm x\|_2 &\le \|M^\dagger\|_2 \| \bm a \|_2 + \frac{\|\bm b\|_2 + \|W^TBM^\dagger\|_2 \|\bm a\|_2}{\sigma_{k-2j}(W^T BQ)} \\
&\le \left(\|M^\dagger\|_2 + \frac{1+\|W^TBM^\dagger\|_2}{\sigma_{k-2j}(W^T BQ)} \right) \sqrt{\|\bm a\|_2^2 + \|\bm b\|_2^2}\\
&\le \left(\|M^\dagger\|_2 + \frac{1+\sigma_{j+1}(BM^{\dagger})}{\sigma_{k-j}(BQ)} \right) \sqrt{\|\bm a\|_2^2 + \|\bm b\|_2^2}.
\end{align*}
\end{proof}

\section{Growth of Gaussian Matrices under Partial Pivoting}\label{sec:partial}

Here, we prove that the growth factor of an $n \times n$ Gaussian matrix $G$ under partial pivoting is rarely much more than $\sqrt{n}$. The proof technique is based on control of most of the singular values of the leading submatrices $\Pi(G)_{[k],[k]}$. Our technique, at a high level, is heavily inspired by that of Huang and Tikhomirov  \cite{huang2024average}, with a number of key refinements and differences in our spectral recursion and handling of extreme singular values. To obtain the correct $\sqrt{n}$ behavior, we exploit the only guaranteed benefit of partial pivoting, that is, that $\|L\|_{\max} =1$. Other than exploiting this key fact, the majority of our analysis is devoted to showing that the greedy nature of partial pivoting does not create a situation which is too much worse than no pivoting at all. While simulations suggest that partial pivoting is in fact better than no pivoting for many of our desired estimates, such as $\Sigma_r(\Pi(G)_{[k],[k]})$, proving this seems significantly more challenging than the growth conjecture itself.

Our argument involves revealing pivot choices or columns of $G$ iteratively and conditioning. Let $\pi_{t}$ be the ordered list of the first $t$ pivot indices, $\pi_t^c = [n]\setminus \{\pi_{t}(1),\ldots,\pi_t(t)\}$ ordered in increasing order, $T_t=G_{\pi_{t},[t]}$, $\mathcal{H}_t$ record the matrix $T_t$ and the pivot indices $\pi_{t}$, and $\mathcal{F}_t$ record the matrix $G_{[n],[t]}$. We say a quantity is measurable based on some information if it can be determined from it. In the coming analysis, we use $\mathcal{H}_t$ to argue about non-pivot rows and $\mathcal{F}_t$ to reveal new columns. Note that many of the results presented are stated in slightly greater generality than needed here, to aid in the smoothed analysis in the following section. Throughout this section, the notation $C_x$, for some variable $x$, denotes a positive constant depending only on $x$, whose value may increase from one occurrence to the next. The exception is a constant used to define a threshold hold, which is fixed once chosen, i.e., $D_\gamma$ below is fixed. For brevity, let $\lambda = \log n$.

\subsection{A High-Level Outline of the Proof} Before turning to the proof itself, we present a brief high-level sketch of the argument. The partial pivoting rule \eqref{eq:partial} implies that $\|L\|_{\max} = 1$, and $\|G\|_{\max}$ is at least one with exponentially high probability. Therefore, it suffices to show $\|U\|_{\max}$ is rarely much bigger than $n^{\frac12}$. We do so by proving an even stronger statement, that is, that $\max_{t} \|G / T_t\|_{\max}$, the largest entry of any of the Schur complements computed during elimination, is rarely much bigger than $n^{\frac12}$. By \eqref{eq:landu}, the entries of $U$ are a subset of the entries of the Schur complements. Let $P_t$ be the permutation matrix that reorders the original row indices as $(\pi_t,\pi_t^c)$, with the pivot indices in pivot order and the remaining indices in increasing order. Let
\[E_t = [ - G_{\pi^c_t,[t]} T_t^{-1} \;  I]P_t \in \R^{(n-t)\times n} \]
be the elimination matrix after $t$ steps of Gaussian elimination, and consider an arbitrary entry of $G/T_t$, corresponding to row $\bm e^T$ of $E_t$ and column index $j>t$ of $G$. Then, conditional on $\mathcal{F}_t$, this entry is normally distributed
\[\bm e^T G_{[n],j} \, | \, \mathcal{F}_t \sim N\big(0,\|\bm e\|_2^2\big),\]
as the $j^{th}$ column of $G$ is independent of $\mathcal{F}_t$. Therefore, controlling the norms of the rows of $E_t$ controls the entries of $G/T_t$ through a union bound over Gaussian variables. Simply put, to prove Conjecture \ref{conj:trefethen}, it suffices to show that, with overwhelming probability, $\|\bm{e}\|_2^2$ is not much bigger than $n$.

Ideally, one could control the row norms of $E_t$ by controlling the singular values of $T_t$. Indeed, control of $\|T_t^{-1}\|_F^2$ would be sufficient, as, conditional on $\mathcal{H}_t$, each row of $G_{\pi_t^c,[t]}$ is distributed as a Gaussian conditioned to lie in the convex set
\begin{equation}\label{eq:convex}
\mathcal{K}_t = \left\{\bm x \, \Big| \, \| \bm x \hat U^{-1}\|_{\infty} \le 1, \; \text{where} \; T_t = \hat L \hat U\right\},
\end{equation}
which, by Lemma \ref{lm:gauss_convex}, retains enough ``Gaussianity'' for our purposes. One might hope that $\|T_t^{-1}\|_F^2 \lesssim n$, which would follow from, say, $\sigma_{t-j}(T_t) \gtrsim (j+1)/\sqrt{n}$ for all $j$. Unfortunately, $\frac1n \|G^{-1}\|_{F}^2$ has a heavy tail and an infinite mean. It is true, however, that, by Lemma \ref{lm:gauss_sing_val}, $\Sigma_r(G_{[t],[t]}) \lesssim n^{1+o(1)}$ for $r = \lceil \sqrt{\lambda} \rceil$ with all but superpolynomially small probability. Unsurprisingly, numerical results suggest that partial pivoting only helps matters, that is, $\Sigma_r(T_t)$ behaves better than $\Sigma_r(G_{[t],[t]})$; see Figure \ref{fig:inverse}. We prove, in Proposition \ref{prop:truncated_bound}, that it is not too much worse. The proof of Proposition \ref{prop:truncated_bound} uses only two properties of partial pivoting: that the $k^{th}$ pivot is chosen after seeing only the first $k$ columns of the matrix, and removing some number of rows not selected in the first $k$ steps and re-running Gaussian elimination without them does not change the first $k$ pivots chosen. In particular, the result is equally valid for any adversarial pivoting scheme with these two properties.

Armed with the information that $\Sigma_r(T_t)$ is not too much more than $n$, we can obtain a bound on the row norms of $E_{t+4r}$ via a rank argument for $E_{t}$ combined with an argument that $4r$ additional steps of Gaussian elimination smooth the $r$ bad directions of $E_t$ without worsening the situation too much in some other way. In Lemma \ref{lm:control}, we show that $E_t$ can be written as the sum of two parts, one whose row norms are controlled by $\Sigma_r(T_t)$ and one of rank at most $r$. In Lemma \ref{lm:fix}, we prove that the randomness introduced by revealing $4r$ additional Gaussian columns smooths out the bad low rank part of $E_t$ and that, by the greedy property of partial pivoting, the elimination matrix for these additional $4r$ steps has row one-norms of size at most $2^{4r}$, which, for $r = \lceil\sqrt{\lambda}\rceil$, are $n^{o(1)}$. This bounds the row norms of $E_{t+4r}$. Using this observation, in Proposition \ref{prop:truncated}, we prove that control on $\Sigma_r(T_t)$ for $t \in [n]$ implies control on the growth factor under partial pivoting.

Combining Proposition \ref{prop:truncated} (control on $\Sigma_r(T_t)$, $t \in [n]$, implies control on growth) with Proposition \ref{prop:truncated_bound} (control on $\Sigma_r(T_t)$, $t \in [n]$) completes the proof.

\begin{figure}[t!]
\centering
  \begin{subfigure}[t]{0.49\textwidth}
    \centering
    \includegraphics[width =.9\textwidth]{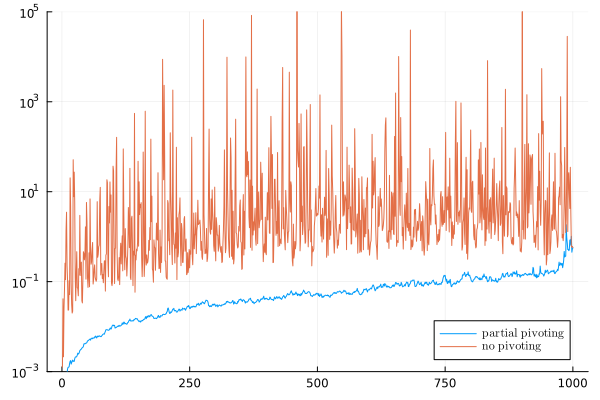}
    \caption{$\frac1n \|T_t^{-1}\|_F^2$ vs $\frac1n\|(G_{[t],[t]})^{-1}\|_F^2$}
  \end{subfigure}\hfill  
  \begin{subfigure}[t]{0.49\textwidth}
    \centering
\includegraphics[width =.9\textwidth]{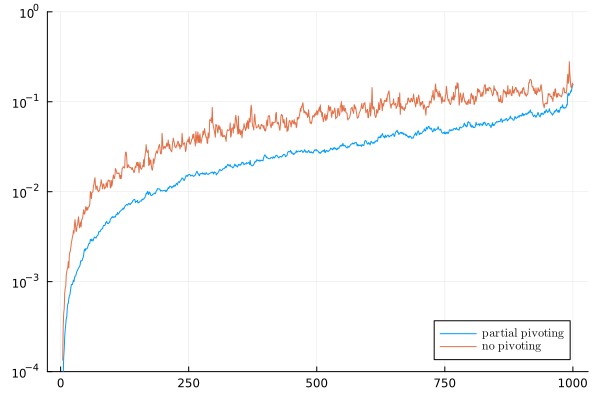}
    \caption{$\frac1n\Sigma_3(T_t)$ vs $\frac1n\Sigma_3(G_{[t],[t]})$}
  \end{subfigure}
 \caption{We compute the singular values of the inverses of the leading submatrices of a single $1000 \times 1000$ Gaussian matrix, both before and after partial pivoting. In Figure (A), the squared Frobenius norms of the inverses without pivoting are poorly controlled. With partial pivoting, the situation is much better, but still this quantity can get quite large near $t = n$. In Figure (B), we consider the Frobenius norms of the truncated inverses, with their $ \lceil \sqrt{\lambda} \rceil = 3$ larger singular values removed. These quantities are much better controlled, both with and without pivoting, though partial pivoting still performs better than no pivoting at all for $t$ not too close to $n$.}
\label{fig:inverse}
\end{figure}

\subsection{Controlling the Truncated Inverse Bounds Growth} To bound the growth factor, it suffices to control the squared Frobenius norms of the inverses of the leading pivot blocks after removing their $r$ largest singular values.

\begin{proposition}\label{prop:truncated} Let $p >0$ and $r< n/8$, and suppose that, for some deterministic $\tau,\eta >0$,
\[\P\left[ \Sigma_r(T_t)\le \tau \text{ for all } t<n\right] \ge 1 - \eta.\]
Then, for all sufficiently large $n$, depending only on $p$,
\[\P\left[ \growth(\Pi(G)) > \sqrt{1 + \tau} \, e^{ C_p \left(r + \frac{\lambda}r\right)}\right] \le \eta + n^{-p}.\]
\end{proposition}

In order to prove our desired proposition, we require the following two lemmas. The first decomposes $E_t$ into a controlled part and a part of rank at most $r$. The second shows that adding $4r$ columns successfully treats the uncontrolled portion, and, for $r$ small enough, does not have enough time to do additional harm during those $4r$ steps.

\begin{lemma}\label{lm:control}
For every $t<n$ and $r\ge 1$, there is a $\mathcal{F}_t$-measurable decomposition
\[E_t = X_t + Y_t Q_t^T,\]
where $Q_t \in \R^{ n \times \min(t,r)}$ satisfies $Q_t^T Q_t = I$. If $t \le r$, every row of $X_t$ has norm one. If $t>r$ and the $\mathcal{H}_t$-measurable condition $\Sigma_r(T_t) \le \tau$ holds, then
\[\P\left[  \max_{i} \|(X_t)_{i,[n]}\|_2^2 > 1 + (2+4x)\tau \, \Big| \, \mathcal{H}_t \right] \le n e^{-x} \qquad \text{for all} \quad x >0.\]
\end{lemma}

\begin{proof}
When $t \le r$, we may simply take $X_t = [0 \; I]P_t$ and put the rest in the rank $r$ part. When $t > r$, consider a singular value decomposition
\[T_t^{-1} = Z_t + H V^T,\]
where $V^TV = I_r$ and $Z_t$ contains the $t-r$ smallest inverse singular values. Now, let
\[ X_t = [ - G_{\pi^c_t,[t]}Z_t \; I ]P_t, \qquad Y_t = - G_{\pi^c_t,[t]} H, \qquad \text{and} \qquad Q_t = P_t^T \begin{bmatrix} V \\ 0 \end{bmatrix}.  \]
We have $\|Z_t\|_2^2 \le \|Z_t\|_F^2 = \Sigma_r(T_t) \le \tau$, and each row of $X_t$ has squared two-norm $1 + \|G_{i,[t]} Z_t\|_2^2$. Note that, conditional on $\mathcal{H}_t$, the unpivoted rows at step $t$ are distributed as a standard Gaussian conditioned to lie in $\mathcal{K}_t$. Applying Lemma \ref{lm:gauss_convex} with zero mean and taking a union bound provides our desired result.
\end{proof}

\begin{lemma}\label{lm:fix}
Let $n - t> 4r$, condition on $\mathcal{F}_t$, and let $E_t = X_t + Y_t Q_t^T$ be as in Lemma \ref{lm:control}. If every row of $X_t$ has two-norm at most $\zeta$, then, for every $q \ge 1$,
\[\max_i \|(E_{t+4r})_{i,[n]}\|_2 \le \zeta \, e^{C_q\left(r+\frac{\lambda}{r}\right)},\]
except with conditional probability at most $C n^{-q}$, where $C$ and $C_q$ do not depend on $\|Y_t\|_2$.
\end{lemma}

\begin{proof}
Let $G_\star = G_{[n],[t+4r]\setminus[t]}$ contain the next $4r$ columns of the matrix, ordered according to $E_t$, and consider the result of $4r$ further steps of elimination applied to
\[E_t G_\star = X_t G_\star + Y_t (Q_t^T G_\star) \in \R^{(n-t)\times 4r}.\]
Note that, conditional on $\mathcal{F}_t$, $Q_t^T G_\star \in \R^{\min(t,r) \times 4r}$ is Gaussian. Let $P$ permute the rows indexed by $\pi_t^c$ so that the next $4r$ pivot rows come first, in pivot order, and the remaining rows follow, in the order of $\pi_{t+4r}^c$. Let the matrix $J \in \R^{(n-t-4r)\times(n-t)}$ encode the row operations performed during these additional $4r$ elimination steps, with its remaining rows returned to the order of $\pi^c_{t+4r}$. In particular, 
\[J = [- (PE_t G_\star)_{[n-t]\setminus[4r],[4r]} ((PE_t G_\star)_{[4r],[4r]})^{-1} \;\, I_{n-t-4r}]P.\]
Notably, $J E_t = E_{t+4r}$ and $J E_t G_\star = 0$, the result of elimination. Therefore,
\[0 =J E_t G_\star = J X_t G_\star + J Y_t (Q_t^T G_\star) = (J X_t G_\star (Q_t^T G_\star)^\dagger + J Y_t) (Q_t^T G_\star),\]
and so 
\[E_{t+4r} = J X_t  +J Y_t Q_t^T = J(X_t -  (X_t G_\star) (Q_t^T G_\star)^\dagger Q_t^T).  \]
Note that we have removed the dependence on $Y_t$ and its size, replacing it by the controlled portion $(X_t G_\star)$ and a Gaussian pseudoinverse $(Q_t^T G_\star)^\dagger$. Because $J$ encodes $4r$ steps of partial pivoting, it is the product of $4r$ matrices, each of whose rows are of the form $\bm e_i - \ell_i \bm e_j$ for $|\ell_i| \le 1$, and so $\|J_{i,[n-t]}\|_1 \le 2^{4r}$ for all rows. Therefore, by the triangle inequality,
\begin{align*}
 \|(E_{t+4r})_{i,[n]}\|_2 &\le \sum_{j=1}^{n-t}|J_{ij}|\|(X_t -  X_t G_\star (Q_t^T G_\star)^\dagger Q_t^T)_{j,[n]}\|_2 \\
&\le 2^{4r}\left( \zeta + \max_{i} \|(X_t G_\star)_{i,[4r]}\|_2 \|(Q_t^T G_\star)^\dagger\|_2 \right)
\end{align*}
for all $i$. Each row of $X_t G_\star$ has independent Gaussian entries with variance at most $\zeta^2$. By Lemmas \ref{lm:gauss_bounds} and \ref{lm:wide_gauss}, with $x = (q+1) \lambda$, we have
\[ \max_{i} \|(X_t G_\star)_{i,[4r]}\|_2^2 \le C_q \zeta^2(4r + \lambda) \qquad \text{and} \qquad \|(Q_t^T G_\star)^\dagger\|_2^2 \le \frac{C_q}{r} e^{C_q \lambda/r}\]
with conditional probability at least $1 - C n^{-q}$. Combining these two bounds proves our desired result.
\end{proof}

Using Lemmas \ref{lm:control} and \ref{lm:fix}, the proof of our desired proposition follows fairly quickly.

\begin{proof}[Proof of Proposition \ref{prop:truncated}]
Let $q = p + 4$ and apply Lemma \ref{lm:control} at every starting stage $t < n - 4r$ with $x = q \lambda$. Then
\[\max_i \|(X_t)_{i,[n]}\|_2^2 \le C_q (1 + \lambda)(1 + \tau)\]
for all such $t$, with probability at least $1 - \eta - n^{2-q}$, where $\eta$ accounts for the truncated inverse bound failing, and $n^{1-q} \times n$ accounts for each of the desired bounds failing. Now, apply Lemma \ref{lm:fix} at each of these starting stages and take a union bound. For $u \le 4r$, the deterministic bound $\|(E_{u})_{i,[n]}\|_2 \le 2^u$ suffices. We have
\[\max_{u<n} \max_i \|(E_{u})_{i,[n]}\|_2  \le \sqrt{1 + \tau} \, e^{C_q\left(r + \frac{\lambda}{r}\right)}\]
with probability at least $1 - \eta - C n^{2-q}$, after absorbing the $\sqrt{1+\lambda}$ into the exponential.
At a fixed stage $t$, the row bound is $\mathcal{F}_t$-measurable and each Schur complement entry divided by its elimination row norm is a standard Gaussian. Applying the standard bound $\P[|X|>h] \le 2 e^{-\frac12 h^2}$ for $X \sim N(0,1)$ for $h = \sqrt{2 q \lambda}$ and union bounding over all the Schur complement entries, which contain all the entries of $U$, we get that
\[ \|U\|_{\max} \le \sqrt{ 1 + \tau } \, e^{C_p\left(r + \frac{\lambda}{r}\right)}\]
with probability at least $1-\eta - C n^{3-q}$. Note that throughout this argument, we take union bounds over stage-wise conditional estimates, and do not use estimates conditioned on success at all stages. 

Finally, by partial pivoting, $\|L\|_{\max} =1$, and $\|G\|_{\max} \ge 1$ with all but an exponentially small probability, completing the proof.
\end{proof}

\subsection{Extending a Selected Block} In order to bound $\Sigma_r(T_t)$, we require a lemma that bounds the event that a leading pivot block has a good truncated inverse but its extension has too many small singular values, which we apply inductively in the following subsection.
\begin{lemma}\label{lm:extension}
Let $m + 4k \le n$, $k<m$, $d\ge 16$, and $k \ge 100d$. For every $\beta \ge 1$, there is a constant $C_\beta$ such that, for every $0<\mu < \sqrt{k}$,
\[\P\left[\sigma_{m-k}(T_m)\ge \mu, \quad  \Sigma_k(T_m) \le 2 k \mu^{-2}, \quad \text{and} \quad 
\sigma_{m+4k-d}(T_{m+4k}) \le \frac{\mu d}{k} e^{-C_\beta \left(1 + \frac{k \lambda}{d^2}\right)} \right] \le n^{-\beta}.\]
\end{lemma}
\begin{proof}
Let $s = 4k$ and suppose the $\mathcal{F}_m$-measurable event 
\[\sigma_{m-k}(T_m)\ge \mu \qquad \text{and} \qquad  \Sigma_k(T_m) \le 2 k \mu^{-2}\]
holds. Conditional on $\mathcal{F}_m$, the pivots $\pi_m$ are fixed and $G_{\pi_m,[m+s]}$ is the fixed block $T_m$ and $s$ additional Gaussian columns. By Lemma \ref{lm:gauss_extension} with $x = (\beta+1) \lambda$,
\begin{equation}\label{eq:extend_estimates}
    \|(G_{\pi_m,[m+s]})^\dagger\|_2^2 \le \mu^{-2} e^{C_\beta\left(1+ \frac{\lambda}{k}\right)} \qquad \text{and} \qquad \|(G_{\pi_m,[m+s]})^\dagger\|_F^2 \le k \mu^{-2} e^{C_\beta\left(1+ \frac{\lambda}{k}\right)}
\end{equation}
with probability at least $1-3n^{-(\beta+1)}$. Note that this concerns the extension of the rows specifically chosen by $\pi_m$, and so this cost is only paid once.

Now, to treat the next $s$ pivot rows, we fix a set $S$, $|S|=s$, and consider the result of $m$ steps of elimination with partial pivoting without the rows indexed by $S$. Let $\pi^{S}_{m}$ be the first $m$ pivots in that situation. Conditional on all the rows outside of $S$, $G_{\pi^S_m,[m+s]}$ is fixed and $G_{S,[m+s]}$ is Gaussian. In particular, whether \eqref{eq:extend_estimates} holds for $G_{\pi^S_m,[m+s]}$ is already determined. We do not condition on $S$ being the next pivot set. Now, on the event that  \eqref{eq:extend_estimates} does hold for $G_{\pi^S_m,[m+s]}$, let the columns of $Q_S$ be an orthonormal basis for $\mathrm{null}(G_{\pi^S_m,[m+s]})$. Using Lemma \ref{lm:gauss_bounds} with $x = (s+\beta+1)\lambda$ and \eqref{eq:extend_estimates},
\[\sigma_{\lfloor d/2 \rfloor +1}(G_{S,[m+s]} (G_{\pi^S_m,[m+s]})^\dagger) \le \sqrt{k} \mu^{-1} e^{C_\beta\left(1+ \frac{\lambda}{d}\right)}\]
with probability at least $1-n^{-s-\beta-1}$. By Lemma \ref{lm:gauss_sing_val} with $\theta = e^{-C_\beta\left(1+\frac{k\lambda}{d^2}\right)}$, $j = \lfloor d/2 \rfloor$ and $C_\beta$ a large enough constant so that $s^{j+1} \theta^{\frac14 j^2} \le n^{-s-\beta-1}$,
\[\sigma_{s-\lfloor d/2 \rfloor}(G_{S,[m+s]}Q_S)\ge \frac{d}{\sqrt{k}}e^{-C_\beta\left(1+\frac{k\lambda}{d^2}\right)}\]
with probability at least $1-n^{-s-\beta-1}$.

Because $2 \lfloor d/2 \rfloor \le d$, $\mu \le \sqrt{k}$, and $k \ge 100 d$, by Lemma \ref{lm:adding_rows},
\begin{equation}\label{eq:singval_rowextend}
\sigma_{m+s-d}\left( \begin{bmatrix} G_{\pi^S_m,[m+s]} \\ G_{S,[m+s]} \end{bmatrix} \right) \ge \frac{\mu d}{k} e^{-C_\beta\left(1+\frac{k\lambda}{d^2}\right)}
\end{equation}
for a sufficiently large constant $C_\beta$, where we have absorbed the factor $e^{C_\beta \frac{\lambda}{d}}$ in the numerator because $\lambda/d \le k \lambda/d^2$. There are at most $n^s$ choices for $S$. Therefore, the probability that some candidate satisfies \eqref{eq:extend_estimates} but fails \eqref{eq:singval_rowextend} is at most $2 n^{-\beta-1}$. Applying this to the actual next $s$ pivot rows $S$, for which $\pi^S_m = \pi_m$, completes the proof.
\end{proof}

\subsection{A Bound for the Truncated Inverse} Ideally, $\sigma_{t-k}(T_t) \gtrsim k/\sqrt{n}$ for $k\ge r$, which would imply $\Sigma_r(T_t) \lesssim n \sum_{k = r}^{\infty} 1/k^2 \lesssim n/r$. The following proposition proves such a bound, with an additional exponential term, due to the additional exponential term in Lemma \ref{lm:extension}.

\begin{proposition}\label{prop:truncated_bound}
For every $\gamma >0$, there exist constants $C_\gamma,n_\gamma$ such that, for $n \ge n_\gamma$, 
\[\max_{k<n} \Sigma_r(T_k) \le \frac{C_\gamma n}{r} e^{C_\gamma \left(\frac{\lambda}{r} +\log \log n \right)}.\]
with probability at least $1-n^{-\gamma}$, simultaneously for all $r \ge \lceil \sqrt{\lambda} \rceil$.
\end{proposition}

\begin{proof}
For an integer $d \ge \lceil \sqrt{\lambda} \rceil$, consider the recursion $k(d) = \max(100\, d,\lceil d^2/\lambda \rceil)$, with $d_0 = d$ and $d_{i+1} = k(d_i)$, and let
\[ h(d) = \sum_{i|d_i<n}\left(1 + \frac{\lambda}{d_i}\right) \qquad \text{and} \qquad g(d) = \frac{d}{\sqrt{n}} e^{-D_\gamma h(d)}.\]
We use a factor of $100$ so that Lemma \ref{lm:extension} applies, and use the quadratic term so that the sequence grows quickly once $d >\lambda$. The quantity $h(d)$ tracks the total exponential loss accumulated over each stage from Lemma \ref{lm:extension}, and $g(d)$ provides the iterated thresholds for singular values coming from Lemma \ref{lm:extension}. Note that $h(d)$ is non-increasing and $g(d)/d$ is non-decreasing in $d$. We aim to prove that
\begin{equation}\label{eq:sing_val_bounds}
\sigma_{t-d}(T_t) \ge g(d) \qquad \text{for all} \quad \lceil \sqrt{\lambda} \rceil \le d <t \le n,
\end{equation}
which implies our desired result. In particular, if \eqref{eq:sing_val_bounds} holds, then, for every $t > r \ge \lceil \sqrt{\lambda} \rceil$,
\[\Sigma_r(T_t)\le \sum_{d=r}^{t-1}\frac{1}{g(d)^2} \le \frac{r^2}{g(r)^2} \sum_{d=r}^{\infty} \frac{1}{d^2} \le \frac{2 r}{g(r)^2} = \frac{2n}{r}e^{2 D_\gamma h(r)}.\]
In addition,
\[h(r) \le C\left(\frac{\lambda}{r}+ \log \log n\right),\]
as $\sum_{i} \lambda/d_i \le \lambda/r \sum_{i \ge 0} 100^{-i} \le 2 \lambda/r$, and $h(d)$ involves at most $O(\log \log n)$ steps, $O(\log \log n)$ for $d_i$ to exceed $\lambda$ and another $O(\log \log n)$ steps for $d_i$ to exceed $n$. 

What remains is to prove \eqref{eq:sing_val_bounds} holds with probability $1- n^{-\gamma}$. We do so by induction on $t$. Let $\beta = \gamma + 2$, $D_\gamma$ be sufficiently large, and $n_\gamma$ be large enough that $\sqrt{\lambda} \ge 16$.

First, suppose that $k(d) \ge t/8$. There are $O(n^t)$ choices of $t$ rows, and so we may bound $\sigma_{t-d}(T_t)$ by union bounding over $G_{S,[t]}$ for all $|S|=t$. In particular, by Lemma \ref{lm:gauss_sing_val} applied $O(n^t)$ times with $\theta = c^{-1} e^{-C_\gamma \left(1 + \frac{\lambda}{d} \right)}$, the probability that 
\[\sigma_{t-d}(T_t)<g(d) = \frac{d}{\sqrt{n}} e^{-D_\gamma h(d)} \le   \frac{cd}{\sqrt{t}} \left(c^{-1} e^{-C_\gamma \left(1 + \frac{\lambda}{d}\right)} \right)\]
is at most
\[\exp\left((t+d+1) \lambda - \frac{d^2}{4}\left[C_\gamma \left(1+\frac{\lambda}{d}\right) - \log \frac{1}{c} \right] \right).\]
By assumption, $t \lambda \le 8 k(d)  \lambda \le C(d \lambda + d^2)$, so this probability is at most $n^{-\beta}$ for $C_\gamma$ sufficiently large.

Now, consider the case where $k(d)<t/8$, and let $u = t - 4k(d)$, so that $u>t/2>k(d)$. If our desired result \eqref{eq:sing_val_bounds} holds at stage $u$, then
\[\sigma_{u-k}(T_u) \ge g(k) \qquad \text{and} \qquad \Sigma_k(T_u) \le \sum_{\ell = k}^{u-1} g(\ell)^{-2} \le 2 k g(k)^{-2}. \]
Note that $g(k) \le k/\sqrt{n} \le \sqrt{k}$, so we may apply Lemma \ref{lm:extension} with $\mu = g(k)$, providing our desired inductive step with all but probability $n^{-\beta}$. Taking a union bound over the $<n^2$ pairs $(t,d)$, the total failure probability is at most $n^2 \times n^{-\beta} = n^{-\gamma}$, completing the proof.
\end{proof}

\subsection{A Proof of Theorem \ref{thm:gauss_pp}} Using Propositions \ref{prop:truncated} and \ref{prop:truncated_bound}, Theorem \ref{thm:gauss_pp} follows quickly.

\begin{proof}[Proof of Theorem \ref{thm:gauss_pp}]
Let $r = \lceil \sqrt{\log n} \rceil$. By Proposition \ref{prop:truncated_bound}, with exponent $\gamma = \alpha +1$,
\[\max_{k<n} \Sigma_r(\Pi(G)_{[k],[k]}) \le n e^{C_\alpha \sqrt{\log n}}\]
with probability at least $1 - n^{-(\alpha+1)}$. Combining this with Proposition \ref{prop:truncated}, with exponent $p=\alpha+1$, gives our desired result.
\end{proof}

\section{A Smoothed Analysis of Partial Pivoting}\label{sec:smooth}

Here, we extend the analysis of previous section to the smoothed setting $A + \epsilon G$, where $A \in \R^{n\times n}$ is fixed and $\epsilon >0$. We prove that the growth factor of such a matrix is at most $n^{O(1)}\epsilon^{-O(\log \log n)}$ with overwhelming probability.

\subsection{Adapting the Argument for General Means}

The argument in Section \ref{sec:partial} carries over to the general mean setting, up to some additional loss. By the scale invariance of the growth factor, we may replace $A + \epsilon G$ by $A/\epsilon + G$, so that our matrix remains unit variance but with a possibly very large mean. Let $d_0 = d$, $d_{i+1} = k(d_i)$, and define
\[f_{K}(d) = \prod_{i>0|d_i<n}\left(1+ \frac{K^2}{d_i}\right)\qquad \text{and} \qquad g_{K}(d) = \frac{d}{\sqrt{n}} \frac{e^{-D_\gamma h(d)}}{f_{K}(d)}.\]
In the proof, we set $K = \epsilon^{-1} \|A\|_2$. Tracking the argument, we lose an additional $f_K(r)^2$ when bounding the truncated inverses (the adaptation of Proposition \ref{prop:truncated_bound}) and an additional $(1+K)^3$ when converting truncated inverse control to control on the growth (the adaptation of Proposition \ref{prop:truncated}). Note that when $\|A\|_2 = 0$, both of these terms are one, recovering the original argument in Section \ref{sec:partial}.

\begin{theorem}\label{thm:smooth_general}
For every $\alpha >0$, there exist constants $C_\alpha,n_\alpha$ such that, for all $n \ge n_\alpha$, every $\Delta \in \R^{n \times n}$, $\epsilon >0$, and every $\lceil \sqrt{\lambda} \rceil \le r < n/8$,
\[\P\left[ \growth (\Pi(\Delta +\epsilon G)) > C_\alpha \left(1 + \frac{\|\Delta\|_2}{\epsilon} \right)^3 \sqrt{\frac{n}{r}} \, e^{C_\alpha \left(r + \frac{\lambda}{r}\right)} f_{\frac{ \|\Delta\|_2}{\epsilon}}(r) \right] < n^{-\alpha}.\]
\end{theorem}

\begin{proof}
Without loss of generality, let $\epsilon =1$. We define $B = \Delta + G$, $\pi_t$ to be the first $t$ pivot indices of $B$, $T_t = B_{\pi_t,[t]}$, and $K = \|\Delta\|_2$. $\mathcal{H}_t$ records $\pi_t$ and $T_t$, and $\mathcal{F}_t$ records $B_{[n],[t]}$. The elimination matrices are given by
\[E_t = [ - B_{\pi_t^c,[t]}T_t^{-1} \;  I]P_t\]
and we use the decomposition $E_t = X_t + Y_t Q_t^T$ from Lemma \ref{lm:control}, where $T_t^{-1} = Z_t + H V^T$, with
\[X_t = [-B_{\pi_t^c,[t]}Z_t \; I]P_t, \qquad Y_t = -B_{\pi_t^c,[t]}H, \qquad \text{and} \qquad Q_t = P_t^T \begin{bmatrix} V \\ 0 \end{bmatrix}.\]

First, we convert the truncated inverse bound into a bound on growth (i.e., modify Proposition \ref{prop:truncated}). Suppose $\Sigma_r(T_t)\le \tau$ for every $t<n$ with probability at least $1-\eta$. We aim to bound growth under this hypothesis. Conditional on $\mathcal{H}_t$, each non-pivoted row in the first $t$ coordinates is $N(\Delta_{i,[t]},I)$ restricted to same origin-symmetric convex set $\mathcal{K}_t$ as before. For $t>r$, on the $\mathcal{H}_t$ measurable event that  $\Sigma_r(T_t)\le \tau$, Lemma \ref{lm:gauss_convex} and a union bound implies that
\[
\P\left[\max_{i \in \pi_t^c} \|B_{i,[t]}Z_t\|_2
> \left(\sqrt{2+4q \lambda}+K\right)\sqrt{\tau} \, \Big| \, \mathcal{H}_t\right]<n^{1-q},\]
and so, for $C_q$ large enough and $R=C_q(1+K^2+\lambda)(1+\tau)$,
\[\P\left[\max_i
\|(X_t)_{i,[n]}\|_2^2>R \, \Big| \, \mathcal{H}_t \right] <n^{1-q}\]
on that same event. Again, for $t \le r$, the rows of $X_t$ have two-norm one, so the bound always holds.  Now, let 
\[B_\star = B_{[n],[t+4r]\setminus[t]}, \qquad G_\star = G_{[n],[t+4r]\setminus[t]}, \qquad \text{and} \qquad \Delta_\star = \Delta_{[n],[t+4r]\setminus[t]},\]
and consider $4r$ additional steps of Gaussian elimination. Each row of $X_t B_\star$ has independent Gaussian entries, with  variance at most $R$ and mean row two-norm at most $K \sqrt{R}$. Also, conditional on $\mathcal{F}_t$,
\[Q_t^T B_\star = Q_t^T \Delta_\star + Q_t^T G_\star,\]
where $Q_t^TG_\star$ is Gaussian. By Lemmas \ref{lm:gauss_bounds} and \ref{lm:wide_gauss}, conditioned on the event $\max_i \|(X_t)_{i,[n]}\|_2^2 \le R$,
\[ \max_{i} \|(X_t B_\star)_{i,[4r]}\|_2^2  \le C_q R(K^2+ r + \lambda) \qquad \text{and} \qquad \|(Q_t^T B_\star)^\dagger\|_2^2 \le C r^{-1} e^{C_q\frac{\lambda}{r}},\]
except with conditional probability at most $C n^{-q}$. Again, let $J$ record the next $4r$ steps of elimination, with the non-pivot rows returned to increasing order. As in Lemma \ref{lm:fix},
\[E_{t+4r} = J(X_t - X_t B_\star (Q_t^T B_\star)^\dagger Q_t^T) \qquad \text{and} \qquad \max_{i} \|J_{i,[n-t]}\|_1 \le 2^{4r}.\]
Taking a union bound, together with the deterministic bound for $u \le 4r$, we have
\[\max_{u<n} \max_i \|(E_u)_{i,[n]}\|_2 \le (1+K)^2 \sqrt{1+\tau} \, e^{C_q \left(r + \frac{\lambda}{r}\right)}\]
with probability at least $1 - \eta - C n^{2 - q}$. For a row $\bm e^T$ of $E_u$ and a column $j>u$, conditional on $\mathcal{F}_u$,
\[\bm e^T B_{[n],j} \sim N(\bm e^T \Delta_{[n],j},\|\bm e\|_2^2).\]
Using the bound $|\bm e^T \Delta_{[n],j}| \le K \|\bm e \|_2$ and applying a Gaussian tail bound over the $O(n^3)$ Schur complement entries loses an additional factor $K + \sqrt{2 q \lambda}$. The same probabilistic argument lower bounding the max norm of the matrix also holds for $\|\Delta + G \|_{\max}$, and still $\|L\|_{\max} = 1$. All together, we obtain
\[\growth(\Pi(\Delta + G)) \le (1+K)^3 \sqrt{1 + \tau} \, e^{C_q \left(r+ \frac{\lambda}{r}\right)}\]
with probability $1-\eta - Cn^{3-q}$ at least. This completes the first half of the proof. 

What remains is to produce a modified truncated inverse bound (i.e., modify Proposition \ref{prop:truncated_bound}). First, we adapt Lemma \ref{lm:extension}. Restricting to the $\mathcal{F}_m$-measurable event $\sigma_{m-k}(T_m) \ge \mu$ and $\Sigma_{k}(T_m) \le 2 k \mu^{-2}$. Lemma \ref{lm:gauss_extension} multiplies both bounds in \eqref{eq:extend_estimates} by $v_k$. As in Lemma \ref{lm:extension}, we fix a set $S$ of size $s=4k$ and condition on all the other rows. Again, $\pi_m^S$ is the first $m$ pivots after deleting $S$, $j = \lfloor d/2\rfloor$, and here the columns of $Q_S$ form an orthonormal basis for $\mathrm{null}(B_{\pi_m^S,[m+s]})$. On the event that $B_{\pi_m^S,[m+s]}$ does satisfy the above two bounds, Lemma \ref{lm:gauss_bounds} with $x = (s+\beta + 1)\lambda$ implies that
\[\sigma_{j +1}(B_{S,[m+s]} (B_{\pi^S_m,[m+s]})^\dagger) \le \sqrt{k} \mu^{-1} v_k e^{C_\beta \left(1+ \frac{\lambda}{d}\right)},\]
except with conditional probability at most $n^{-s-\beta-1}$. The small singular value estimate for $B_{S,[m+s]}Q_S$ is unchanged, as Lemma \ref{lm:gauss_sing_val} is unchanged by a non-zero mean. Together, the new output threshold for Lemma \ref{lm:extension} is
\[\frac{\mu d}{kv_k} e^{-C_\beta \left(1+\frac{k \lambda}{d^2}\right)}.\]
The quantity $f_K(r)$ tracks the accumulation of the additional term $v_k$ in the denominator. Take $\beta = \gamma +2$ and choose $D_\gamma$ sufficiently large. Because $f_{K}(d)$ is non-increasing in $d$, $g_{K}(d) /d$ is non-decreasing still. The same argument holds for the product recurrence $k(d)<n$, now with
\[\frac{g_{K}(d)}{g_{K}(k)} = \frac{d}{k v_k} e^{-D_\gamma \left(1+\frac{\lambda}{d}\right)},\]
which leads to the worse bound
\[\max_{t<n} \Sigma_r(T_t) \le \frac{2n}{r} e^{2 D_\gamma h(r)}f_K(r)^2 \le  \frac{C_\gamma n}{r} e^{C_\gamma \left(\frac{\lambda}{r}+\log \log n\right)} f_{K}(r)^2\]
with probability at least $1-n^{-\gamma}$.

The result follows by combining both arguments, with $\gamma = \alpha + 1$, $\eta = n^{- \gamma}$, and $q = \alpha + 4$.
\end{proof}

\subsection{Proof of Theorem \ref{thm:smooth_pp}}

Our desired Theorem \ref{thm:smooth_pp} follows immediately from Theorem \ref{thm:smooth_general}.

\begin{proof}[Proof of Theorem \ref{thm:smooth_pp}]
Consider Theorem \ref{thm:smooth_general} with $\Delta = A$ and $r = \lceil e^4 \lambda \rceil$, for which $r + \lambda/r = O(\lambda)$. Let $n_\alpha$ be large enough so that $r<n/8$. Let $m$ be the last index for which $d_m<n$ in the chain starting with $d_0=r$. Because $d_{i+1} \ge d_i^2/\lambda$, 
\[\log(d_i/\lambda) \ge 2^i \log(r/\lambda) \ge 2^{i+2}.\]
Therefore, $2^{m+2}< \lambda$ and $m + 2 < \log (\lambda)/\log 2$. Let $K = \epsilon^{-1}\|A\|_2$. There are $m$ factors in $f_{K}(r)$, each at most $2 \epsilon^{-2}$, and so
\[\left(1 + K\right)^3 f_{K}(r) \le 2^{m+3}\epsilon^{-(2m+3)} \le 2 \lambda \epsilon^{-(2/\log 2) \log \lambda}. \]
\end{proof}

\section*{Acknowledgements}
The author thanks Louisa Thomas for improving the style of presentation, Alan Edelman for many conversations on all aspects of Gaussian elimination, Nick Trefethen and Konstantin Tikhomirov for conversations regarding the average-case analysis of partial pivoting, and Rikhav Shah for conversations regarding orthogonal preconditioning. The author also thanks Nick Trefethen for comments on an earlier unpublished draft. This material is based upon work supported by the National Science Foundation under Grant No. DMS-2513687.

\section*{AI Usage Statement}
The research presented here is the result of human–AI interaction involving ChatGPT Sol 5.6 \& Astra 6. For example, the proof of Conjecture \ref{conj:trefethen} is the result of a human-only proof sketch for the bottom right $o(n)$ entries of $U_{ij}$, which was improved to a proof for all entries through AI interaction. The other results presented here follow a similar trajectory: a combination of some human-only ideas and expertise which was improved through AI interaction. The results of Sections \ref{sec:orthogonal} and \ref{sec:gauss}, though quite interesting to the author, are not technically difficult, and likely did not need human input. There is reason to suggest, however, that the proof of Conjecture \ref{conj:trefethen} was resilient to standard use of an AI system. Conjecture \ref{conj:trefethen} was listed on the well-known GitHub repository OpenProblemsInNLA, for which roughly half of the $>200$ open problems have been solved, largely with autonomous AI systems \cite{townsend2026openproblemsnla}. Conjecture \ref{conj:trefethen} was given a designation of open, challenging, and interesting to the community and remained unsolved prior to this manuscript. The author welcomes detailed questions regarding the role of AI in this work.  The entire manuscript was written by the author.

{ \small 
	\bibliographystyle{plain}
	\bibliography{main.bib} }

\end{document}